\documentclass{article}

\usepackage{float}
\usepackage{bbm}
\usepackage{amsthm}
\usepackage[utf8]{inputenc}
\usepackage[english]{babel}
\usepackage{color}
\usepackage{amsmath,amssymb,amsthm,yhmath,amsfonts,mathrsfs,enumitem,mathtools,dsfont}
\usepackage{tikz}
\usetikzlibrary{shapes,arrows,positioning}

\newcommand{\pr}{\mathbb{P}}
\newcommand{\E}{\mathbb{E}}

\renewcommand{\d}{\mathrm{d}}

\newcommand{\W}{\mathbb{W}}
\newcommand{\1}[1]{{\mathds{1}}_{\left\{#1\right\}}}

\newtheorem{theorem}{Theorem}

\newtheorem{remark}{Remark}
\newtheorem{lemma}{Lemma}
\newtheorem*{lemma*}{Lemma}
\newtheorem{prop}{Proposition}
\newtheorem{corollary}{Corollary}
\newtheorem{heuristics}{Heuristics}
\newtheorem{claim}{Claim}

\newcommand{\rem}[1]{{#1}}

\author{
	D\'aniel Keliger\\
	{\small Department of Stochastics, Institute of Mathematics,}\\
	{\small Budapest University of Technology and Economics}\\
	{\small M\H{u}egyetem rkp. 3., H-1111 Budapest, Hungary;
	}\\
	{\small Alfr\'ed R\'enyi Institute of Mathematics, Budapest, Hungary}\\
	{\small e-mail: zunerd@renyi.hu}}

\title{Universality of SIS epidemics starting from small initial conditions
	\footnote{Supported by the ÚNKP-21-1 New National Excellence Program of the Ministry for Innovation and Technology from the source of the National Research, Development and Innovation Fund. 
	\includegraphics[width=0.05\textwidth]{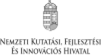}
	\includegraphics[width=0.05\textwidth]{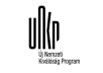}
} 
\footnote{ Partially supported by the ERC Synergy under Grant No. 810115 - DYNASNET.}
}

\begin{document}

\maketitle

\begin{abstract}
	We are investigating deterministic SIS dynamics on large networks starting from only a few infected individuals. Under mild assumptions we show that any two epidemic curves -- on the same network and with the same parameters -- are almost identical up to time translation when initial conditions are small enough, regardless of how infections are distributed at the beginning. The limit object -- an epidemic starting from the infinite past with infinitesimally small prevalence -- is identified as the nontrivial eternal solution connecting the disease  free state with the endemic equilibrium.  Our framework covers several benchmark models including the N-Intertwined Mean Field Approximation (NIMFA) and the Inhomogeneous Mean Field Approximation (IMFA).
\end{abstract}

\section{Introduction}
\label{s:introduction}

As it has been recently shown by the SARS-CoV-2 pandemic, understanding the propagation of diseases is a crucial task for interconnected societies. Making predictions about how much capacities hospitals will need, how effective lockdown measures are, and what is the optimal protocol for distributing vaccines are just few of the questions faced by policy makers.

Such decisions can be supported by epidemiological models where the consequences of certain policies can be simulated without the economic and societal costs of large-scale experimentation. These models usually group individuals into compartments (such as susceptible, infected, recovered, exposed) who make random transitions between such compartments with transition rates depending on the state of individuals being connected to her on a contact network. Examples are the SIS model with two states: \emph{susceptible} (S) and \emph{infected} (I), where infected individuals are cured at a constant rate and susceptible individuals are getting infected at a rate proportional to the number of their infected neighbors. Another example is the SI model, where no curing events happen, making it more suitable to model e.g. the diffusion of information.

Although these models provide flexibility, they have certain shortcomings arising from the high dimensionality of the problem. The exact stochastic model with a population of $n$ individuals has a state space of size $2^n$, making direct calculations infeasible, invoking the need for Monte Carlo simulations or certain reductions. Besides computational limitations, obtaining detailed information about the entire contact network or the complete initial configuration of infections can pose a challenge as well. 

A common way to mitigate the problem of dimensionality is to apply mean field approximations at several scales (and levels of accuracy). Examples from less to more detailed methods are: the Homogenous Mean Field Approximation (HMFA), where it is assumed that the population is well mixed \cite{kurtz70}, the Inhomogenous Mean Field Approximation (IMFA), which keeps track of the degree distribution and assumes individuals are statistically equivalent in the same degree class \cite{vesp} and the N-intertwined Mean Field Approximation (NIMFA) which treats the population at the level of individuals keeping the full contact network and only neglecting dynamical correlations between individuals \cite{NIMFA2011}. One may also apply so-called metapopulation models where people are grouped into smaller communities -- cities, regions etc. \cite{SIS_metapopulation}.

These mean-field models yield ODE systems where each equation describes the evolution of one smaller community or individual. For example, in the HMFA case, only one equation is needed, while NIMFA works with $n$ equations, which is still much less than the $2^n$ equations needed for the exact stochastic process. In this work, we utilize a recent graphon-based approach resulting in a PDE \cite{graphonSISnoise, graphonSIScontrol} which enables studying continuous populations together with all the previously mentioned mean-field approximations.

We start from the observation that real world pandemics usually start from only a few infected individuals, motivating the study of solutions from small but non-zero initial conditions. Clearly, there are some discrepancies between such curves as an epidemic starting from infection ratio $1 \%$ needs less time to saturate than an epidemic starting from $0.01 \%$, resulting in a delay. However, after accounting for this time translation, it turns out that these curves are remarkably close to each other regardless of the initial distribution of infections. We dub this phenomenon the Universality of Small Initial Conditions (USIC).

\begin{figure}[t]
	\centerline{\includegraphics[scale=0.75]{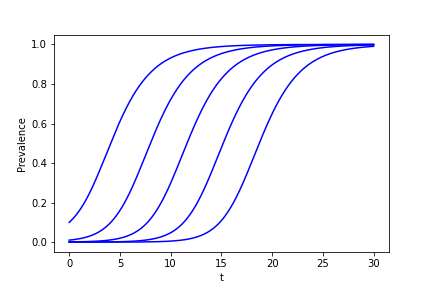}}
	\caption{The ratio of infected individuals for an SI epidemics (parameters: infection rate $\beta=1$, curing rate $\gamma=0$) on a power-law network with parameter $p=0.4$ starting from a ratio of infected individuals $10^{-1}, \dots, 10^{-5}$ at time 0. (For further details, see Section \ref{s:notations}.)}
	\label{fig_curves}
\end{figure} 

Based on USIC, one can also expect a certain limit object to arise. At time $t=0$ set the ratio to be an arbitrary value between $0$ and equilibrium prevalence, say, $10 \%$. If the initial infection is small, it should originate from a large negative $t<0$ value, hence, in the limit there is an epidemic curve starting from the infinite past from infinitesimal amount of infections. We call such limit the \emph{nontrivial eternal solution} --which is unique up to time translation -- as it is defined for all past future values as well.
\begin{figure}[t]
	\centerline{\includegraphics[scale=0.75]{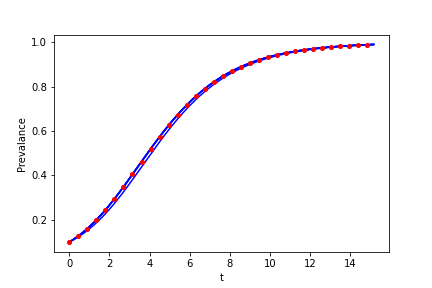}}
	\caption{Solid lines: The epidemic curves from Figure \ref{fig_curves} after a time shift. Dotted line: The eternal solution given by \eqref{eq:SI_explicit1}. }
	\label{fig_translate}
\end{figure} 

Roughly speaking, USIC states that -- under the same network and parameters -- there is only one "relevant" epidemic curve (up to translation), and it is given by the nontrivial eternal solution, which is robust to perturbations of the initial condition. It is worth emphasizing though that the nontrivial eternal solution is only relevant until the underlying parameters: the network, the infection and curing rates remain constant. In real life, this can be violated by imposing lockdown measures, vaccination, mutations and the changing awareness of the population. Thus, we only expect the nontrivial eternal solution to characterize the early phase of the epidemic where assuming constant parameters is plausible.

We will show (see \eqref{eq:small_eternal2}) that the nontrivial eternal solution at the beginning increases exponentially in time and infections are distributed according to the leading eigenvector of the operator corresponding to the linearized PDE \eqref{eq:v} which is nothing but the eigenvector centrality when the graph is finite. It is worth mentioning that the eigenvalue centrality can be estimated from samples when the graph is dense enough \cite{graphonL2conv}.

In summary, at the early stage of the epidemic -- shortly after the leading eigenvector starts dominating the dynamics -- one can assume the parameters to be constant and the prevalence to be small. To make predictions, the full initial configuration is not needed, only the ratio of infected. Furthermore, there is no need for the full contact network, instead, the eigenvector centrality is sufficient, which can be estimated \rem{even when only a sample of the underlying network is known  \cite{graphonL2conv}.} This results in a major reduction of dimensionality and robust predictions.

\rem{
The main idea behind the phenomena is the following: When the number of infections are small the system can be linearized. If we have a positive spectral gap the leading eigenvector $\varphi_1$ will dominate and the solutions will have the form $ce^{\alpha_1 t} \varphi_1$ where $c$ is a constant and $\alpha_1$ is the growth rate. The difference in the constant $c$ for two different initial conditions can be controlled by a time shift after which both solutions have the form $\varepsilon\varphi_1+o(\varepsilon).  $ If we converge to the stationary solution in a uniform fashion, then the we need $O \left( \frac{1}{\varepsilon} \right)$ amount of time during which the error becomes $o \left( \frac{1}{\varepsilon} \cdot \varepsilon \right)=o(1).$

Note that this argument implicitly assumes that the linearization breaks down only \emph{after} the largest eigenvector starts dominating the linear system. This might require some stronger assumptions how well the underlying network is connected.

In this paper we restrict our attention to the SIS model, but we expect the argument above to hold for wider class of models such as SIR or SEIR. 
}

The paper is structured in the following manner: The rest of Section \ref{s:introduction} lists the related works. Section \ref{s:notations} introduces the notion of graphon kernels and the corresponding PDE describing the SIS process on said kernel. Results are stated in Section \ref{s:results}.Proofs are given in Section \ref{s:proof}. Finally Section \ref{s:outlook} gives a brief summary along with possible directions for future research.

\subsection*{Related works}

\rem{Although, some articles articles already discuss phenomena similar to USIC, to the authors knowledge, this is the first work in the setting given in Section \ref{s:notations} and making the connection between USIC and the nontrivial eternal solution. Furthermore, the examples given below only shows USIC-like behavior for global quantities such as the total number of infections while we provide results for local infection densities as well enabling a considerable dimension reduction.

We would also like to mention that while the linearization argument is far from being novel, the fact that linearization breaks down after the principal eigenvector starts dominating and the error remains small in the later phase of the process as well is something that has not been studied before in the epidemic contest to the authors knowledge. 
}

There is a huge body of research for eternal solutions of reaction-diffusion systems. \cite{reaction_eternal} is an example where the eternal solution connects two stationary solutions reminiscent of our setup. For the SIS model on $\mathbb{R}$, \cite{SIS_PDE_on_R} constructs certain eternal solutions. \rem{In this article, instead, we focus on inhomogenous networks and make a connection between eternal solutions and USIC.  }

A diagram similar to Figure \ref{fig_curves} appears in the work of Volz \cite{volz} \rem{at page 20} studying the stochastic SIR process on graphs generated by the configuration model. \rem{The cause of time translation seems to be the stochastic nature of the initial phase which could be approximated be a branching process as in \cite{SIR_USIC, SIR_USIC_Juli} discussed below, rather than due to qualitatively different initial conditions.   }

Volz's approach turns out to be exact in the large graph limit \cite{volzproof,volz_proof_readable}. Furthermore, in Theorem 2.9 of \cite{volz_proof_readable}, the authors show that starting an epidemic with $o(\rem{N})$ infections will converge to an ODE conditioning on large outbreaks. This result implicitly contains USIC as time being translated to start from $s_0\rem{N}$ susceptible vertices after which all (surviving) processes follow the same deterministic dynamics given by the ODE. Note that the role of small initial conditions is relevant as the ODE uses the degree distribution of the susceptible vertices, which is roughly the same as the global degree distribution when there are only $o(\rem{N})$ infections initially.

\cite{SIR_USIC, SIR_USIC_Juli} consider a stochastic SI model on configuration model graphs starting from a single infected individual. A notable difference is that they allow non-Markovian transitions and construct the stochastic time delay obtained from the martingale of the corresponding branching process.

\cite{kiss2020} studies the stochastic SIS process and approximates it by a birth-death process. The number of SI links is estimated for a given number of infected vertices on a ``typical'' trajectory. This is achieved empirically in two scenarios: first, the epidemic starts from one initially infected vertex to get the segment between $0$ and the endemic state (``left side''); in the second scenario, an epidemic starts with everyone being infected to get the segment from the endemic state to $\rem{N}$ (``right side''). We conjecture that the left side of the curve corresponds to the nontrivial eternal solution in the limit $\rem{N} \to \infty$ as it represents a ``typical'' trajectory starting from small initial conditions.

\section{Main concepts and notations}
\label{s:notations}

\subsection{Graphons and kernels}
In this section we describe contact patterns between individuals represented by graphon kernels. For a more detailed overview on graphons see \cite{lovaszbook}. 

Individuals are represented by the variable $x \in [0,1]$. For example, in the finite case, with vertices indexed from $1$ to $n$, individual $i$ can be represented with the value $x=\frac{i}{n}.$

Connections are given by a symmetric kernel $W:[0,1]^2 \mapsto \mathbb{R}^{+}_{0}$  in $L^2 \left([0,1]^2 \right).$ One can interpret $W(x,y)$ as the probability of $x,y \in [0,1]$ being connected, or the strength of their interaction.

We will make use of the following \emph{connectivity property}: For all measurable $A \subseteq [0,1]$ with both $A$ and $A^c$ having positive measure one has
\begin{align}
\label{eq:W_connectivity}
\int_{A} \int_{A^c} W(x,y) \d x \d y>0.
\end{align}

Due to $W \in L^2 \left([0,1]^2 \right) $, the corresponding integral operator
\begin{align*}
\W f(x):=\int_{0}^{1} W(x,y) f(y) \d y
\end{align*}
is a Hilbert-Schmidt operator. If we also assume that $W$ is \emph{irreducible} -- that is, for some integer $r$, the iterated kernel $W^{(r)}$ is positive a.e. -- then there is a leading eigenvalue
\begin{align}
\label{eq:spectral_gap}
\forall i \neq 1 \ \lambda_1>|\lambda_i|
\end{align}
and the corresponding eigenvalue $\varphi_1$ is positive a.e. (see \cite{generalized_Jentzsch}). Also, we can choose an orthonormal eigenvector basis $(\varphi_k)_{k=1}^{\infty}$ and express $W$ as
\begin{align*}
W(x,y)=\sum_{k=1}^{\infty} \lambda_k \varphi_k(x)\varphi_k(y).
\end{align*}

Besides the leading eigenvalue $\lambda_1$, we pay special attention to $\lambda_2,$ the second largest eigenvalue (not the second largest in absolute value). $\lambda_1-\lambda_2>0$ is referred to as the \emph{spectral gap}.

Note that irreducibility and \eqref{eq:W_connectivity} trivially holds when $W$ is uniformly positive, that is $\exists m_0>0 \ : \ W(x,y)\geq m_0$ a.e. In all of our examples we either assume the existence of such $m_0$ or study piece-wise constant $W$ corresponding to a finite, weighted graph (see Section \ref{s:discrete} ). In the later case we assume that said weighted graph is \emph{connected} (or equivalently \eqref{eq:W_connectivity}) under which condition the Perron Frobenius is applicable making all the previous claims true save from \eqref{eq:spectral_gap} which is modified to
\begin{align}
\forall i \neq 1 \ \lambda_1>\lambda_i.
\end{align}

Furthermore, for all of our examples, $\varphi_1$ is uniformly positive, that is
\begin{align}
\label{eq:unifpos}
\exists m>0 \ \forall x \in [0,1] \quad \varphi_1(x) \geq m.
\end{align}

\subsection{Finite and annealed graphs}
\label{s:discrete}
We call the kernel $W$ discrete if there is a non-degenerate finite  partition $(I_i)_{i=1}^{n}$ of $[0,1]$  and non-negative constants $\left(W_{ij}\right)_{i,j=1}^n$ such that
\begin{align*}
W(x,y)=\sum_{i=1}^{n}\sum_{j=1}^n W_{ij} \1{I_i}(x) \1{I_j}(y)	
\end{align*}
meaning, $W$ is piece-wise constant. 

\rem{Note that in metapopulation models we might make a distinction between $N$, the total number of individuals and $n$,the number of subpopulations (say cities). The formalism allows for $N=n$ when one "subpopulation" represents and individual. }

For functions in the form
$$f(x)=\sum_{i=1}^{n} f_i \1{I_i}(x)$$ 
applying the integral operator leads to
\begin{align*}
(x \in I_i) \ \W f(x)=\sum_{j=1}^n W_{ij} \int_{I_j} f(y) \d y=\sum_{j=1}^n W_{ij}|I_j| f_j
\end{align*}
motivating the definition of the (asymmetric) weights $w_{ij}:=W_{ij} |I_j|$ and the matrix $\mathcal{W}:=(w_{ij})_{i,j=1}^{n}.$

We will always assume that $\mathcal{W}$ is connected in the sense that for all $i,j$ there are $k_1,\dots, k_l$ such that $w_{ik_1}w_{k_1k_2} \dots w_{k_{l-1}k_l}w_{k_lj}>0.$  

Setting $I_{i}=]\frac{i-1}{n},\frac{i}{n} ]$ (with $I_1=[0,\frac{1}{n}]$) gives a finite weighted graph with $n$ vertices and weights given by the (now symmetric) matrix $\mathcal{W}.$ In particular, setting $\mathcal{W}=\frac{1}{n}A$ to a (rescaled) adjacency matrix leads to a classical, unweighted graph, with \eqref{eq:W_connectivity} describing its connectivity.

Another interesting case covered is annealed networks \cite{annealed}. Let $k_1,\dots, k_n$ denotes the finite number of possible degrees of a graph. $p(k_i)$ is the ratio of vertices having degree $k_i$, and $p(k_i|k_j)$ refers to the probability of a stub of degree $k_j$ connecting to a stub of degree $k_i$. In the special case of uncorrelated networks, $p(k_i|k_j)=\frac{k_i p(k_i)}{\langle k \rangle}$ where $\langle k \rangle$ is the average degree.

The interval corresponding to the degree class $k_i$ is set to be 
$$I_i=]\sum_{j=1}^{i-1}p(k_{j}),\sum_{j=1}^{i-1}p(k_{j})+p(k_i)]$$
(with $I_1=[0,p(k_1)]).$ Naturally, $|I_i|=p(k_i).$

The weights are $w_{ij}=k_i p(k_j|k_i)$.

Finally, we have to check that $W_{ij}=\frac{w_{ij}}{|I_j|}$ is symmetric so that the weight matrix $\mathcal{W}$ indeed corresponds to a discrete $W$ kernel.
\begin{align*}
W_{ij}=&\frac{k_i p(k_j|k_i)}{p(k_j)}=\langle k \rangle \frac{p(k_i,k_j)}{p(k_i)p(k_j)},
\end{align*}
where $(2-\delta_{k_i,k_j})p(k_i,k_j)$ is the probability that a randomly chosen edge connects two vertices with degree $k_i$ and $k_j$, hence $W_{ij}$ is indeed symmetric. 

In general, $\mathcal{W}$ is symmetric with respect to the scalar product $[f,g]:=\sum_{i=1}^{n}f_i g_i \frac{1}{|I_i|},$ hence all the eigenvalues are real. It is easy to see that for any $\phi$ eigenvector of $\mathcal{W},$ $\varphi(x)=\sum_{i=1}^n \phi_i \1{I_i}(x)$ is an eigenvector of $\W.$ The other way round, $\phi_i:=\frac{1}{|I_i|}\int_{I_i} \varphi(x) \d x$ also makes $\phi$ an eigenvector. In particular, the Perron–Frobenius vector $\phi^{(1)}$ of $\mathcal{W}$ corresponds to $\varphi_1(x)=\sum_{i=1}^n \phi_i^{(1)} \1{I_i}(x),$ making $\varphi_1$ uniformly positive with $m:=\min_{1 \leq i \leq n} \phi_i^{(1)}>0$ in the discrete case.

\subsection{Rank-$1$ kernels}

The kernel $W$ has rank $1$ if it has the form $W(x,y)=\lambda_1 \varphi_1(x)\varphi_1(y).$ Note that uncorrelated annealed graphs are included in this class as $W_{ij}=\frac{k_i k_j}{\langle k \rangle} .$ The parameters are
\begin{align*}
\phi_i^{(1)}=&\frac{1}{\sqrt{\langle k^2 \rangle}}k_i\\
\lambda_1=& \frac{\langle k^2 \rangle}{\langle k \rangle },
\end{align*}
where $\langle k^2 \rangle$ is the second moment of the degrees.

An important example to keep in mind is the rank-$1$ graphon with eigenfunction $\varphi_1(x)=\sqrt{1-2p}x^{-p}:$
\begin{align}
\label{eq:powerlaw}
W(x,y)=\lambda_1 (1-2p)x^{-p}y^{-p} \ \ 0 \leq p < \frac{1}{2}.
\end{align}
The condition $0 \leq p < \frac{1}{2}$ is needed so that $\varphi_1$ is an $L^2\left( [0,1]\right)$ function. 

The corresponding degree function is
\begin{align*}
d(x):= \int_{0}^{1} W(x,y) \d y=\lambda_1 \frac{1-2p}{1-p}x^{-p}	
\end{align*}
hence the ``degree distribution'' of the graphon has a power law decay:
\begin{align*}
\pr\left(d(U) \geq x \right)= \pr \left( U \leq d^{-1}(x) \right)=\left(\lambda_1 \frac{1-2p}{1-p} \right)^{\frac{1}{p}}x^{-\frac{1}{p}},
\end{align*}
fore large enough $x$ where $U$ is a uniform random variable on $[0,1].$

Note that in this example, $\varphi_1(x)$ satisfies the uniform positivity assumption \eqref{eq:unifpos} with $m=\sqrt{1-2p}.$

\subsection{The PDE dynamics}

The (deterministic) SIS process on the graphon kernel $W$ is described by the PDE
\begin{align}
\label{eq:u}
\partial_t u=\beta(1-u)\W u-\gamma u.
\end{align}
with parameters $\beta>0, \gamma \geq 0$. 

\rem{\cite{SISdyn} studies the basic properties of a PDE describing the SIS dynamics on general community structures of which \eqref{eq:u} is a special case. The ideas presented there are heavily utilized in our proofs. Such systems were also studied in \cite{graphongame}. Furthermore, it is know to be the limit of the exact, stochastic SIS process on a sequence of dense enough graphs \cite{Delmas2023, Kuehn2022,Keliger2022}.  }

Here, $u(t,x)$ corresponds to the probability that individual $x \in [0,1]$ being infected at time $t,$ or, in the case of metapopulation models, the ratio of infected within population $x$. Recovery happens at a constant rate $\gamma$, and a susceptible individual $x$ becomes infected at rate
$$ \beta \W u(t,x)=\beta \int_{0}^{1} W(x,y)u(t,y) \d y, $$
where the integral corresponds to the number of its infected neighbours.

When $\gamma>0$, one can set $\gamma=1$ via time change. The $\gamma=0$ case corresponds to the SI process, where no curing is allowed. The SI model is more suitable to describe the diffusion of information rather than viral infections.

Although \eqref{eq:u} describes a deterministic dynamics, it has plenty of connections to the stochastic SIS process.

\rem{When $W$ is \emph{discrete}, \eqref{eq:u} gives back certain well-known mean field approximations:
\begin{align}
	\label{eq:z}
	\frac{\d}{\d t}z_i(t)=\beta(1-z_i(t))\sum_{j=1}^{n}w_{ij}z_j(t)-\gamma z_j(t)
\end{align} 
referred to as  NIMFA or quenched mean field approximation in the literature. 

The relationship between \eqref{eq:u} and \eqref{eq:z} is the following. If $u(t,x)$ is the solution of \eqref{eq:u} then	
\begin{align}
\label{eq:u_and_z}
z_i(t):=\frac{1}{|I_i|} \int_{I_i}u(t,y) \d y
\end{align}
solves \eqref{eq:z}.

On the other hand, if $(z_i(t))_{i=1}^n$ solves \eqref{eq:z} then
\begin{align}
\label{eq:z_and_u}
u(t,x):=\sum_{i =1}^n z_i(t) \1{I_i}(x)
\end{align}
gives a solution to \eqref{eq:u}.

Note that this kind of projection loses some local information regarding the infections within the sets $I_i$, however, this information is either something we do not care about or can not measure in the case of metapopulation models, or does not have any physical interpretation in the case of finite (potentially weighted) graphs where the sets $I_i$ represents individuals. Hence, we only care about solutions which are constant on $I_i$ of the form \eqref{eq:z_and_u}.

(Such reductions are only approximate when $W$ is continuous instead of blockwise constant.) 
}
It has been shown that NIMFA gives an upper bound on the exact infection probabilities \cite{simon2017NIMFA,Mieghem2014}, and for large average degrees, the error can be arbitrarily small \cite{Sridhar_Kar1, Sridhar_Kar2,NIMFA_sajat} \rem{even in the $N=n$ case when $z_i(t)$ represents infection probabilities of individuals rather than the ratio of infections within a subpopulation.}

When $n=1$, without loss of generality, one can set $w_{11}=1$, deriving HMFA:
\begin{align*}
\frac{\d}{\d t} z(t)=\beta(1-z(t))z(t)-\gamma z(t),
\end{align*}
which is shown to be the exact limit of the prevalence on complete graphs \cite{kurtz70}.

To get IMFA \cite{vesp} one must use annealed networks with weights $w_{ij}=k_ip(k_j|k_i).$
\begin{align}
\label{eq:IMFA}
\begin{split}
\frac{\d}{\d t}z_i(t)=&\beta k_i (1-z_i(t))\Theta_i(t)-\gamma z_i(t)\\
\Theta_i(t)=& \sum_{j=1}^{n} p(k_j|k_i)z_j(t)
\end{split}
\end{align} 

Smooth kernels $W$ can also arise as the limit of the discrete system \eqref{eq:z} when the underlying graph converges to $W$ in the graphon sense \cite{lovaszbook,graphonL2conv}. In certain special cases with diverging average degree, the stochastic SIS model converges to \eqref{eq:u}, making \eqref{eq:u} an exact limit, not just an approximation \cite{gl}.

Not all solutions of \eqref{eq:u} are physically relevant as $u(t,x)$ must retain a probabilistic interpretation. Hence we restrict our attention to the domain
\begin{align}
\label{eq:Delta}
\Delta:=\{ \left. f \in L^2([0,1]) \right| 0 \leq f(x) \leq 1 \ a.e. \}.
\end{align}

\begin{prop}
\label{t:PDE_basics}
Assume $0 \leq W \in L^2\left([0,1]^2\right)$ and $u_0 \in \Delta.$ Then among the solution to \eqref{eq:u} satisfying $\forall t \in \mathbb{R}_{0}^{+} \ u(t) \in \Delta$ there is is a unique $u(t)$ with $u(0)=u_0$. 
\end{prop}
\rem{The proof can be found in Section \ref{s:PDE_proofs}.}

Proposition \ref{t:PDE_basics} has already been proven in \cite[Proposition 2.9. (ii)]{SISdyn} under the assumption that $W$ is bounded -- in which case some further smoothness properties also follow.Extending it to square-integrable $W$ is relatively straightforward using the following approximation result:
\begin{prop}
\label{t:approx}
Let $0 \leq W_1,W_2 \in L^2\left([0,1]^2\right)$ such that $\|W_1 \|_2, \| W_2\|_2 \leq \lambda $ for some $\lambda \geq 0$. Let $u_1,u_2$ be two solutions of \eqref{eq:u}  such that $\forall t \in \mathbb{R}_{0}^{+} \ u(t) \in \Delta.$

Then 
\begin{align}
\label{eq:approx}
\sup_{0 \leq t \leq T} \|u_1(t)-u_2(t)\|_2 =O\left(\|u_1(0)-u_2(0) \|_2+\|W_1-W_2 \|_2 \right)
\end{align}
with constants in the $O(\cdot)$ notation depending only on $T$ and $\lambda$.
\end{prop}
\rem{The proof can be found in Section \ref{s:PDE_proofs}.}

\rem{
\begin{remark}
\label{r:discrete}
An easy consequence of Proposition \ref{t:approx} is that we can approximate \eqref{eq:u} with the finite ODE system \eqref{eq:z}.

Let $W_n$ be a discrete approximation of $W$. $u_n,u(t)$ are the solution of \eqref{eq:u} for the graphons $W_n,W$ respectively with initial conditions $u_n(0)=u(0)$. Clearly,
$$\sup_{0 \leq t \leq T} \|u_n(t)-u(t) \|_2=O \left(\|W-W_n\|_2 \right).$$

We can choose a sequence $W_n \overset{L^2([0,1])}{ \to} W$ as simple functions are dense in $L^2([0,1])$ and measurable sets can be approximated by rectangles.

When $W$ is continuous one can take $$W_n(x,y)=\sum_{i,j=1}^n W_{ij}^{(n)}\1{I_i^{(n)}}(x)\1{I_j^{(n)}}(y)$$
where $W_{ij}^{(n)}=W \left(\frac{i}{n},\frac{j}{n}\right) \int_{I_i^{(n)}}\int_{I_j^{(n)}}W(x,y) \d x \d y$. When $W$ is also Lipschitz continuous we further attain $\|W-W_n\|_2 =O \left( \max_{1 \leq i \leq n} \left|I_i^{(n)}\right| \right).$    

\end{remark}
}
Note that  $u(t) \in \Delta$ is only required for $t\geq 0$, and it might not be true for past times $t<0$. For example, let $A \subseteq [0,1]$ be such that both $A,A^c$ have positive measure and $u(0,x)=\1{A^c}(x).$ Then, based on \eqref{eq:W_connectivity},
\begin{align*}
\left.\frac{\d}{\d t} \int_{A}u(t,x) \d x \right |_{t=0}=\beta \int_{A} \W u(0,x) \d x=\beta \int_{A} \int_{A^c} W(x,y) \d y \d x>0
\end{align*}
making $\int_{A}u(-t,x) \d x<0$ for some small $t > 0.$

We call a solution $u(t)$ \emph{eternal} if it satisfies
$$u(t) \in \Delta\quad \forall t \in \mathbb{R}.$$

Obviously, stationary solutions are always eternal. In the supercritical case ($\beta \lambda_1>\gamma$, see Section \ref{s:linear}) for bounded $W$ there are two such stationary solutions \cite[Proposition 4.13.]{SISdyn}: the unstable disease free state $u(t) \equiv 0$ and a nonzero endemic state $\psi$ satisfying
\begin{align}
\label{eq:psi}
\beta(1-\psi)\W \psi=\gamma \psi.
\end{align}
$\psi$ can be also expressed in the form
\begin{align*}
	\psi(x)=\frac{\beta \W \psi(x)}{\gamma+\beta \W \psi(x)}.
\end{align*}

\begin{remark}
When $W$ is rank-$1$, the nonzero solution of \eqref{eq:psi} is
\begin{align*}
\psi(x)=&\frac{\beta \lambda_1 c \varphi_1(x)}{\gamma+ \beta \lambda_1 c \varphi_1(x)}\\
0<c:=&\langle  \varphi_1, \psi \rangle.
\end{align*}
Note that $c$ must solve
\begin{align*}
	1=& \int_{0}^{1} \frac{\varphi_1^2(x)}{\frac{\gamma}{\beta \lambda_1 }+c\varphi_1(x)} \d x.
\end{align*}

Hence the nontrivial solution of \eqref{eq:psi} is unique for rank-$1$ graphons too with $\psi \in \Delta$. Later, in Lemmas \ref{l:convergence_to_equilibrium} and \ref{l:epsilon_0_tilde} we will see that $u(t) \to \psi$ holds in this slightly more general case as well when $u \not \equiv 0.$ 
\end{remark}

When $\gamma>0$ we will use the function
\begin{align}
	\label{eq:pi}
	\pi(x):= \left( 1-\psi(x)\right)^{-1}=1+\frac{\beta}{\gamma} \W \psi(x).
\end{align}
Clearly,
\begin{align*}
	1 \leq \pi(x) \leq 1+\frac{\beta}{\gamma} \int_{0}^{1} W(x,y) \d y
\end{align*}
in particular the $\|\pi \|_1 < \infty$.

The weighted $L^2$ space $L^2_{\pi}([0,1])$ includes measurable functions such that
\begin{align}
\label{eq:pi_norm}
\|f \|_{\pi}^2:= \int_{0}^{1} \pi(x) f^2(x) \d x<\infty. 
\end{align} 
Clearly, bounded functions are included in $L^2_{\pi}([0,1])$. Note that since $1 \leq \pi$
\begin{align*}
 \| f\|_2 \leq \|f \|_{\pi}.
\end{align*}

We call an eternal solution $u(t)$ \emph{nontrivial} if it is neither the disease free, nor the endemic state.

When considering discrete $W$, we have to further restrict the physically viable solutions to
\begin{align}
	\label{eq:Delta_I}
	\Delta_I:=\{ \left. f \in \Delta \right| \forall \ 1 \leq i \leq n \ f|_{I_i} \textit{ is constant.} \}
\end{align}
Note that for discrete $W$, $\Delta_I$ is forward invariant, hence \eqref{eq:u_and_z} \rem{and \eqref{eq:z_and_u}} creates a one-to-one map between the solutions of \eqref{eq:u} and \eqref{eq:v} on the domain $\Delta_{I},$ thus, for discrete $W$, we will work with $\Delta_I$ instead of $\Delta$ in this paper.
 
\subsection{Linearization of the PDE}
\label{s:linear}
The expansion of $u(t)$ by the orthonormal basis $\left(\varphi_k\right)_{k=1}^{\infty}$ is written as
\begin{align}
\label{eq:u_expansion}
u(t)&=\sum_{k=1}^{\infty}c_k(t)\varphi_k\\
\label{eq:c_k}
c_k(t)&:= \langle u(t), \varphi_k \rangle. 
\end{align}

We linearize \eqref{eq:u} around $u=0$ to get
\begin{align}
\label{eq:v}
\begin{split}
\partial_t v&= \mathcal{A}v:=\left(\beta \W-\gamma\mathbb{I} \right)v \\
v(t_0)&=u(t_0).
\end{split}
\end{align}

The initial time $t_0$ will vary in the theorems and proofs, however, it will be clear from the context which version of $v$ we are referring to.

The corresponding expansion of $v(t)$ with respect to $\left( \varphi_k \right)_{k=1}^{\infty}$ is
\begin{align*}
v(t)&=\sum_{k=1}^{\infty}\tilde{c}_k(t)\varphi_k \\
\tilde{c}_k(t)&:= \langle v(t), \varphi_k \rangle.	
\end{align*}

One of the interpretation of \eqref{eq:v} is as the expectation of a Branching random walk. We put a Poisson measure on $[0,1]$ with intensity $v(t_0,x) \d x$. Then an individual at position $x$ creates an offspring on sight $y$ at rate $\beta W(x,y) \d y$ and dies at rate $\gamma$. Then the expected individuals around site $x$ at time $t$ is described by $v(t,x)$. This fact along with Lemma \ref{l:u_and_v} (see Section \ref{s:auxiliary}) implies
\begin{align*}
 \forall t \geq t_0 \quad 0 \leq u(t) \leq v(t).
\end{align*}

Note that $\varphi_k$ are eigenvectors to $\mathcal{A}$ as well with eigenvalues
\begin{align*}
\alpha_k=\beta \lambda_k-\gamma.
\end{align*}

To solution to the linear system \eqref{eq:v} can be written as
\begin{align*}
v(t)=e^{\mathcal{A}(t-t_0)}v(t_0).
\end{align*}
Since $\mathcal{A}$ is bounded, the exponential can be understood as a power series and
\begin{align}
\label{eq:v_expanded}
v(t)&= \sum_{k=1}^{\infty} c_k(t_0)e^{\alpha_k(t-t_0)} \varphi_k. 
\end{align}
Note that $\tilde{c}(t_0)=c_k(t_0).$

It is easy to see that there is a phase transition for $v(t)$ at $\beta_c= \frac{\gamma}{\lambda_1}.$ When $\beta<\beta_c,$ $v(t) \to 0$ for all initial conditions. However, when $\beta>\beta_c,$ the leading term $c_1(t_0)e^{\alpha_1(t-t_0)}\varphi_1$ survives if $c_1(t_0)=\int_{0}^{1} u(t_0,x)\varphi_1(x) \d x>0.$ $\varphi_1>0$ a.e. makes $\varphi_1 \d x$ being equivalent to the Lebesgue measure, therefore, $c_1(t_0)>0$ is equivalent to $\| u(t_0)\|_1>0,$ meaning, we are not considering the disease free epidemic.

From now on, it is assumed that we are in the supercritical case, that is, $\beta \lambda_1>\gamma.$

\section{Results}
\label{s:results}
In this section we present the main results of the article.

\subsection{USIC}

The precise statement for the universality of small initial conditions is the following:
\label{ss:usic}

\begin{theorem} (Main)
\label{t:main}

Assume $W \in L^2([0,1]^2)$ is non-negative along with the dynamics being supercritical: $\beta \lambda_1>\gamma.$

Further assume a), b) or c) holds:
\begin{itemize}
	\item{a)} $W$ is discrete and connected,
	\item{b)} $0<m_0 \leq W(x,y) \leq M$ and $\gamma<\beta \lambda_1<\gamma+2\beta(\lambda_1-\lambda_2), $
	\item{c)} $W$ is rank-$1$, $\varphi_1$ is uniformly positive and in $L^{2+\rho}([0,1])$ for some $\rho>0$.
\end{itemize}

Then, for all $\varepsilon, \eta>0$ there is a $\delta>0$ such that for all $u_1(0),u_2(0) \in \Delta$ (in the discrete $W$ case $u_1(0),u_2(0) \in \Delta_I$) with $0<\| u_1(0)\|_2, \|u_2(0)\|_2 \leq \delta$ there are time shifts $t_1,t_2 \geq 0$ such that
\begin{align}
\label{eq:USIC_main_part}
\sup_{t \geq 0} \|u_1(t+t_1)-u_2(t+t_2) \|_2 \leq \varepsilon,
\end{align}
while
\begin{align}
\label{eq:USIC_second_part}
(i=1,2) \ \sup_{0 \leq t \leq t_i} \| u_i(t)\|_2 \leq \eta.
\end{align}
\end{theorem}

The proof can be found in Section \ref{s:USIC_proof}.

Note that the c) is satisfied for power law kernels \eqref{eq:powerlaw}.

\eqref{eq:USIC_second_part} requires some explanation.  \eqref{eq:USIC_main_part} without \eqref{eq:USIC_second_part} is meaningless as both $u_1(t)$ and $u_2(t)$ converge to the endemic state as $t \to \infty$, hence, we could choose $t_1,t_2$ to be large enough so that $u_{i}(t_i) \approx \psi$ ($i=1,2$) satisfying the \eqref{eq:USIC_main_part} in a trivial manner. \eqref{eq:USIC_second_part} mitigates this problem by asserting that the solutions are small even after the time shifts, hence they can be close to each other even from an early stage of the epidemic.

\rem{
	In condition b) the condition $\alpha_1=\beta \lambda_1-\gamma< 2 \beta (\lambda_1-\lambda_2)$ ensures we have a large enough spectral gap compared to the growth rate ensuring the domination of the leading eigenvector before the linearization breaks down. We conjecture the condition  $\lambda_1>\lambda_2$ to be enough, however, some complications may arrise when $\varphi_1$ is allowed to be zero as it might be difficult to compare $c_1(0)=\int_{0}^{1} u(0,x) \varphi_1(x) \d x$ and $\|u(0) \|_2$ without Lemma \ref{l:c1_bound}.   
} 

\rem{
\begin{remark}
Not that even in the discrete case the results are mostly suitable for dense enough graphs - with diverging average degrees - generated by stochatic block models or $W$-random graphons. We imagine first fixing $u_1(0),u_2(0)$ and $W$ then letting the number of vertices $N$ to infinity in which case the stochastic dynamics on the graph $G_N$ is described by \eqref{eq:u}.

In this setting, $\delta$ does not depend on $N$, the number of vertices on the finite graph, but it might depend on $n$, the number of subpopulations on the discrete $W$ (say the number of cities in a metapopulation model.) Although, the formalism allows for $N=n$, the distinction is more important when comparing the stochastic and the deterministic models. 

Rigorous study of the dependence of $\delta_n$ on some sequence of discrete $W_n$ is out of the scope of this article, however, we believe $\delta$ might need to be unreasonably small when the discrete graph $W$ is too sparsely connected, limiting the range of applicability. This has a 
\end{remark}
}

\rem{ Using \eqref{eq:v_expanded} we expand on the heuristic given in the Introduction.}
\begin{heuristics}
\label{h:main}
When a solution $u(t)$ is small, it can be approximated by the linear system \eqref{eq:v}. As we can see from \eqref{eq:v_expanded}, the leading eigenvector will dominate after some time and
\begin{align*}
u(t) \approx v(t) \approx c_1(0)e^{\alpha_1 t} \varphi_1.	
\end{align*}
Here, we implicitly assume such dominance happens before the linerized system stops being accurate.

Setting $t_i:=\frac{1}{\alpha_1} \log \left( \frac{\varepsilon}{\langle u_i(0), \varphi_1 \rangle} \right) \ (i=1,2)$ gives
\begin{align*}
u_i(t_i)=\varepsilon \varphi_1+o(\varepsilon).	
\end{align*}

We need $\tau=\frac{1}{\alpha_1}\log \left(\frac{1}{\varepsilon} \right)$ time till we get to a constant level while we expect the error to increase by a factor of $e^{\alpha_1 \tau}=\frac{1}{\varepsilon}$ making $\| u_1(t_1+\tau)\|_2,\|u_2(t_2+\tau) \|_2=\Theta(1)$ while $\|u_1(t_1+\tau)-u_2(t_2+\tau) \|_2=o(1). $

For times $t >\tau$ $u_1(t_1+t),u_2(t_2+t)$ remain close together as both of them converge to $\psi$ as $t\to \infty.$
\end{heuristics}

\rem{
\begin{remark}
Note that based on Heuristics \ref{h:main} we expect the time shift to be
\begin{align}
\label{eq:timeshit}
t_1-t_2=-\frac{1}{\alpha_1} \log \frac{\langle u_1(0),\varphi_1 \rangle}{\langle u_2(0),\varphi_1 \rangle}
\end{align}
which is independent of $\varepsilon$. This is the time shift we are using at Figure \ref{fig_USIC_metapop2} where it seems to work quiet well.
\end{remark}
}
To summarize, when two initial conditions are small enough -- but not identically $0$ -- we can apply an appropriate time shift after which the two solutions remain close to each other.

\subsection{Basic properties of eternal solutions}
The proofs for the statements of this section can be found in Section \ref{s:proof_eternal}.

Consider the following heuristics.

\begin{heuristics}
Take a sequence of initial conditions $u_n(-t_n) \to 0$ with $-t_n \to -\infty$ and at time $t=0$ take some intermediate value between the disease free and the endemic state, say  $\|u_n(0)\|_1=\frac{1}{2}\| \psi \|_1.$

Since $u_n(-t_n)$ is getting smaller and smaller, USIC suggest that the solutions $u_n(t)$ are becoming more and more similar to each other, hence there should be a limit $u_n \to u$ by the Cauchy-argument. This limit should satisfy $\lim_{t \to -\infty}u(t)=0, \ \lim_{t \to \infty}u(t)=\psi$ while $u(t) \in \Delta$ for all $t \in \mathbb{R}.$ Since $\|u(0)\|_1=\frac{1}{2}\| \psi \|_1,$ the limit object differs from the disease free and the endemic state, hence, it describes a nontrivial eternal solution making them a natural limit object to study.
\end{heuristics}

\begin{theorem}(Existence of nontrivial eternal solutions)
\label{t:existence}	

Assume $0 \leq W \in L^2([0,1]^2)$ is irreducible (or in the discrete case, connected) together with  supercriticality and $\varphi_1 \in L^{2+\rho}([0,1])$ for some $\rho>0$. Then, there is a nontrivial eternal solution such that
\begin{align}
\label{eq:small_eternal}
\lim_{t \to -\infty} \frac{\langle \varphi_1, u(t) \rangle}{\| u(t)\|_2}=1.
\end{align}

Also, when $W$ is discrete, the nontrivial eternal solution takes values from $\Delta_I.$
\end{theorem}	

Note that in Theorem \ref{t:main} $\varphi_1 \in L^{2+\rho}([0,1])$ is either assumed explicitly or a consequence of $W$ being bounded.

The significance of \eqref{eq:small_eternal} is that it allows us to describe the shape of $u(t,x)$ in the early stages of the epidemic, where the eternal solution is most applicable. Since initially there are only a few infections (see Lemma \ref{l:convergene_to_0} below), $u(t)$ can be approximated by \eqref{eq:v_expanded}, and while the weights are mostly concentrated on $\varphi_1$ due to \eqref{eq:small_eternal}. Therefore, after appropriate time translation we get
\begin{align}
\label{eq:small_eternal2}
u(t,x) \approx \|u(0) \|_2 e^{\alpha_1 t} \varphi_1(x)
\end{align}
when $\|u(0) \|_2$ is small.

\begin{lemma}
\label{l:convergene_to_0}
Assume $0 \leq W \in L^2([0,1]^2)$ together with the connectivity property \eqref{eq:W_connectivity} and supercriticality. Let $u(t)$ be a nontrivial eternal solution. Then $\lim_{t \to -\infty}u(t)=0$ a.e.
\end{lemma}

Next we turn to uniqueness of the eternal solution. Of course, any uniqueness can only be up to time translation, since if $u(t)$ is an eternal solution, then any time translated version $u(t+\tau)$ will also be an eternal solution.

Without the connectivity assumption of Lemma \ref{l:convergene_to_0}, several fundamentally different solutions may arise. For example, when the graph is the union of two disjoint complete graphs, we could treat the solutions on them separately, hence a mixture of disease free state on one component, endemic state on the other is possible.

However, under USIC, the only ambiguity that can occur is due to time translation. Combining it with Lemma \ref{l:convergene_to_0} \emph{the} nontrivial eternal solution can be interpreted as an epidemic started from the infinite past from an infinitesimally small initially infected population with. Also, the beginning exhibits exponential growth with "spatial" distribution described by $\varphi_1(x)$ due to \eqref{eq:small_eternal2}.

Since epidemics usually starts from only a small amount of initial infections, USIC shows that the nontrivial eternal solution is the limit object, reducing the problem to a simple curve from the original infinite dimensional problem, at least for given parameters $\beta,\gamma,W$.

\begin{theorem}(Uniqueness of nontrivial eternal solutions)
\label{t:uniqunes}

Assume the conditions of Theorem \ref{t:main}. Let $u_1,u_2$ be two nontrivial eternal solutions. Then there is a translation time $\tau$ such that $u_1(t+\tau)=u_2(t).$
\end{theorem}

\subsection{Explicit formulas and approximations}

Here, we shows some explicit formulas and heuristics regarding the eternal solution.

\subsubsection{Heuristics for infections close to criticality}

So far we mostly gave implicit descriptions of the nontrivial eternal solutions, apart from \eqref{eq:small_eternal2}. In this section, we aim to give some more explicit formulas for some cases.

We highlight the work \cite{NIMFA_time_dependent} deriving a $\tanh(t)$-formula
\begin{align*}
	\frac{\d}{\d t} c(t)=&(\beta \lambda_1-\gamma )c(t)\left(1-c(t) \right)\\
	u(t,x) \approx &(\beta \lambda_1-\gamma)c(t)\varphi_1(x)
\end{align*}
 for epidemics on finite networks when the infection rate $\beta$ is just slightly above critical. Intuitively, when we are close to criticality, the endemic state becomes small, hence one can linearize \eqref{eq:psi}  to get $\beta \W \psi \approx \gamma \psi$ making $\psi \approx (\beta \lambda_1-\gamma) \varphi_1.$ This means that the initial growth phase \eqref{eq:small_eternal2} looks similar to the saturation phase around $\psi$, resulting in a logistic curve with ``spatial'' distribution given by $\varphi_1(x).$

We believe these ideas can be generalized for a wider class of $W$ kernels \rem{besides finite graphs.}

\subsubsection{Explicit formulas for rank-$1$ kernels}

In then special case when the kernel has the form $W(x,y)=\lambda_1 \varphi_1(x)\varphi_1(y)$ it is possible to give a more explicit construction for the nontrivial eternal solution. At the cost of modifying $\beta$ we can set $\lambda_1=1$ without loss of generality. 

There are two qualitatively distinct cases: one where $\gamma=0$ corresponding to an SI dynamics and one with $\gamma>0$ for the general SIS dynamics.

\subsubsection*{The SI case}

When $\gamma=0$ we can rescale time to set $\beta=1$ resulting in
\begin{align*}
\partial_t u=(1-u)\W u.	
\end{align*}

Let $u(t)$ be a nontrivial eternal solution. Treating $\W u(t)$ as a known function, \
\begin{align*}
u(t,x)=1-\exp\left( -\int_{t_0}^{t} \W u(s,x) \d s\right)(1-u(t_0,x)).
\end{align*}

From Lemma \ref{l:convergene_to_0} $u(t_0) \to 0$ as $t_0 \to -\infty$ resulting in
\begin{align*}
u(t,x)=1-\exp\left( -\int_{-\infty}^{t}  \W u(s,x) \d s\right).
\end{align*}

Since $W$ is rank-$1$
\begin{align*}
\int_{-\infty}^{t}  \W u(s,x) \d s=\varphi_1(x)\int_{-\infty}^{t} c_1(s) \d s=:\Omega(t)\varphi_1(x),
\end{align*}
which yields
\begin{align}
\label{eq:SI_explicit1}
u(t,x)=1-e^{-\Omega(t)\varphi_1(x)}
\end{align}
successfully separating the temporal and spatial variables. Our goal now is to express the temporal part $\Omega(t).$

Define the function
\begin{align}
\label{eq:F}
F(\omega):=\int_{0}^{1} \varphi_1(x)\left(1-e^{-\omega \varphi_1(x)} \right) \d x.
\end{align}
Note that for $\omega> 0$, $F$ is positive and for $\omega \geq 0$, it is Lipschitz continuous with constant $1$. It is also easy to see that $\Omega(0)>0$ and $\Omega(t)$ is increasing.

Observe $\frac{\d}{\d t} \Omega(t)=c_1(t)=\langle \varphi_1,u(t) \rangle,$ resulting in the dynamics
\begin{align}
\label{eq:Omega_ODE}
\frac{\d}{\d t} \Omega(t)=F(\Omega(t)).
\end{align}

To solve \eqref{eq:Omega_ODE} define
$$G(\omega):=\int_{0}^{\omega} \frac{1}{F(\omega')} \d \omega'.$$
Note that $G$ is strictly increasing, hence invertible, leading to
\begin{align}
\Omega(t)=G^{-1}\left(t+G(\Omega(0)) \right).
\end{align}

\subsubsection*{The SIS case}

By rescaling time one can set $\gamma=1$ while $\beta>1$ as we need supercriticality.

Define $u^c(t,x):=1-u(t,x)$ referring to the probability of being susceptible.
\begin{align*}
&\partial_t u^c(t)=-\beta u^c(t) \W u(t)+(1-u^c(t)) \\
&\partial_t u^c(t)+\left(1+\beta \W u(t) \right)u^c(t)=1 \\
&u^c(t,x)=\exp \left(-(t-t_0)-\beta\int_{t_0}^{t} \W u(s,x) \d s \right) \left[u^c(t_0,x) \right. \\
& \left.+\int_{t_0}^{t} \exp\left((s-t_0)+\beta \int_{t_0}^{s} \W u(\tau,x) \d \tau\right) \d s \right]
\end{align*}

Notice $0 <\exp \left(-(t-t_0)- \beta\int_{t_0}^{t}\W u(s,x) \d s \right) \leq e^{-(t-t_0)} \to 0$ as $t_0 \to -\infty$. Therefore, the first term disappears in the limit.

\begin{align*}
u^c(t,x)=&\int_{-\infty}^{t}e^{-(t-s)} \exp \left(-\beta \int_{s}^{t} \W u(\tau,x) \d \tau \right) \d s= \\
 &\int_{0}^{\infty} e^{-s} \exp \left(-\beta \int_{t-s}^{t} \W u(\tau,x) \d \tau \right)  \d s\\
 u(t,x)=&\int_{0}^{\infty} e^{-s}\left[ 1-\exp \left(-\beta \int_{t-s}^{t} \W u(\tau,x) \d \tau \right) \right] \d s
\end{align*}
Note that $\int_{t-s}^{t} \W u(\tau,x) \d \tau=\varphi_1(x)\int_{t-s}^{t} c_1(\tau) \d \tau=\varphi_1(x) \left[\Omega(t)-\Omega(t-s) \right]$ making
\begin{align}
\label{eq:SIS_explicit1}
u(t,x)=\int_{0}^{\infty}e^{-s}\left[1-\exp \left(-\beta \left[\Omega(t)-\Omega(t-s) \right] \varphi_1(x) \right) \right] \d s.	
\end{align}

The dynamics for the temporal part becomes the following Delayed Differential Equation (DDE)
\begin{align}
\label{eq:Omega_DDE}
\frac{\d}{\d t} \Omega(t)=\int_{0}^{\infty} e^{-s}F \left(\beta \left[\Omega(t)-\Omega(t-s) \right] \right) \d s.
\end{align}

\subsection{Simulations}
\label{s:simulation}

Here we present some simulations illustration the results of Section \ref{s:results}.

\rem{Firstly, we discus Figure \ref{fig_curves} and \ref{fig_translate} form Section \ref{s:introduction} and Figure \ref{f:high_low}.

As stated under the description of Figure \ref{fig_curves}, $\beta=1, \ \gamma=0$ the graphon is rank-$1$ with power-law with parameter $p=0.4$. So the PDE is
\begin{align*}
\partial_t u(t,x)=0.2x^{-0.4}(1-u(t,x)) \int_{0}^{1} y^{-0.4} u(t,y) \d y.
\end{align*}

For the numerical approximation of the PDE \eqref{eq:u} we first discretize $W$ according to Remark \ref{r:discrete} with $n=100$ and also cut off the power-law function at the value $1000$. Then, ODE system \eqref{eq:z} is numerically integrated via Euler's method with step size $h=0.003$. The global ratio of infections are plotted.

For Figure \ref{fig_curves} and \ref{fig_translate}  The initial conditions are $z_i(0) \equiv=10^{-k} \ k=1,\dots, 5.$ For Figure \ref{f:high_low} there are 3 cases: uniform initial condition (blue solid line) where $z_i(0) \equiv 0.01$, low degree initial infections (red line with circles) with $z_n(0)=1, \ z_i(0)=0$ for $i<n$ and high degree initial infections where $z_1(0)=1, \ z_i(0)=0$ for $i>1$. (Note that $\varphi_1\left( \frac{i}{n} \right)=\sqrt{1-2p} \left( \frac{i}{n} \right)^{-p}$ is monotone decreasing in $i$, thus high degree nodes have small index $i$. )

In both cases time shift is numerically found  where the values reach (approximately) $0.1$.
\begin{figure}[H]
	\centerline{\includegraphics[scale=0.50]{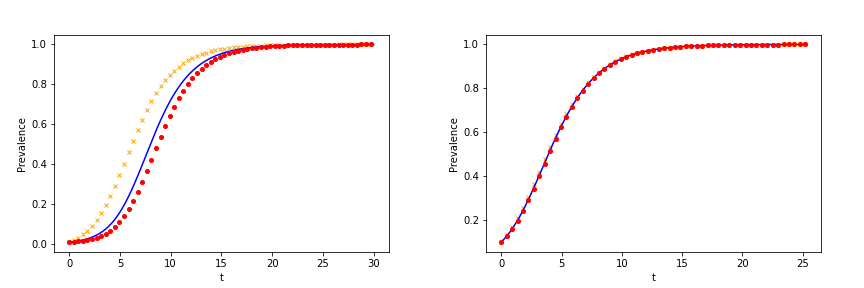}}
	\caption{The ratio of infected individuals for an SI epidemics (parameters: infection rate $\beta=1$, curing rate $\gamma=0$) on a power-law network with parameter $p=0.4$. Initially $1$\% of the population is infected who are distributed in the following three ways: Uniform infection (blue solid line), low degree infection (red circle line) and high degree infection (orange cross marked line). Left: before translation. Right: after translation. }
	\label{f:high_low}
\end{figure}

 }

For the rest of the simulations the parameters are fixed at $\beta=10, \ \gamma=1$. The population is made of five communities labeled from A to E with size and edge density given by Figure \ref{fig_metapop_graph}. In total there are $N=1500000$ individuals. \rem{With numbers, the vector $(|I_i|)_{i=1}^{5}$ is $[\frac{5}{15}, \frac{2}{15}, \frac{2}{15}, \frac{3}{15}, \frac{3}{15}]$ and the matrix $\left(W_{ij} \right)_{i,j=1}^{5}$ is
\[
\begin{bmatrix}
		1 & 0.25 & 0.25 & 0 & 0.1 \\
		0.25 & 1 & 0.25 & 0& 0 \\
		0.25 & 0.25 & 1 & 0.5 & 0 \\
		0 & 0 & 0.5 & 1 & 0.25 \\
		0.1 & 0 & 0 & 0.25 & 1 \\
\end{bmatrix}.
\]

For the time shift in Figure \ref{fig_USIC_metapop2} uses \eqref{eq:timeshit}. 

For the stochastic simulations we use the Gillespie algorithm.

}
The PDE in this case reduces to NIMFA given by \eqref{eq:z}. Note that NIMFA only requiers the relative size of the sub-populations given by $|I_i|,$ the absolute size $N |I_i|$ is only relevant for the stochastic simulations.

\begin{figure}[H]
	\centerline{\includegraphics[scale=0.50]{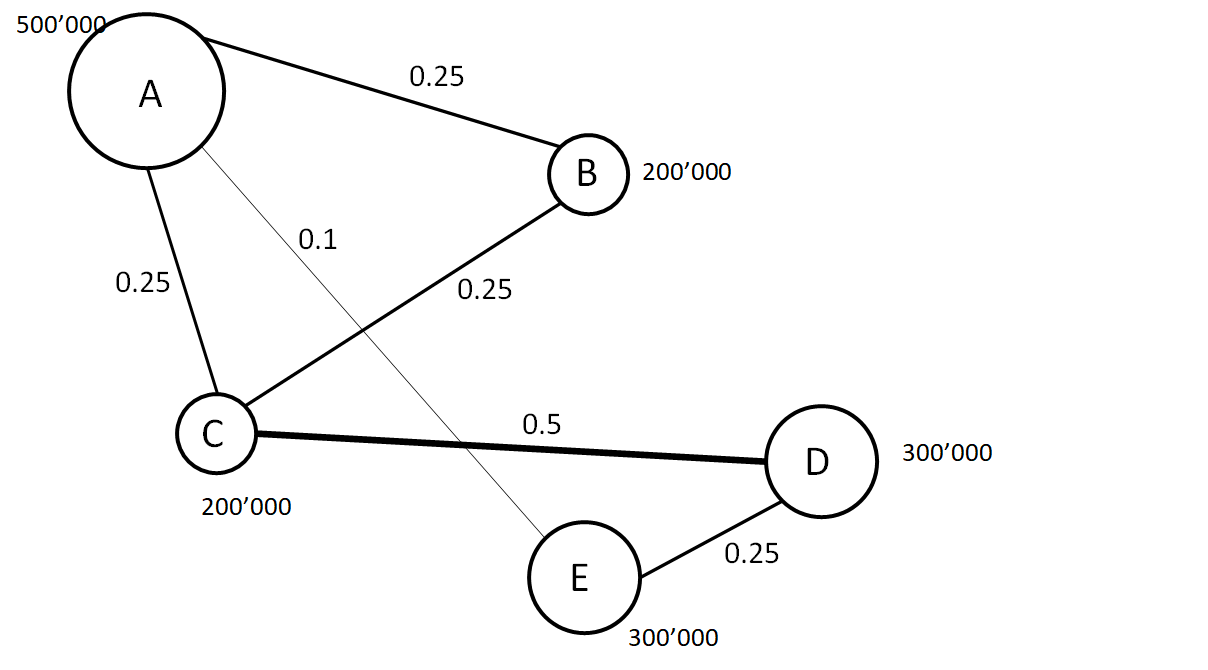}}
	\caption{The meta-population graph for the simulations. The numbers on nodes and vertices denotes the size of the sub-population ($N |I_i|$) and edge density ($W_{ij}$) respectively. Within population edge densities are set to $W_{ii}=1$.  }
	\label{fig_metapop_graph}
\end{figure}

In the first setting there are $100$ infected individuals starting from community E. The $30$ stochastic simulations and the numerical solution for \eqref{eq:z} runs until $T=3$.

We compare the proportion of infected individuals both in the sotochasti and the deterministic simulations to the leading eigenvector $C\varphi_1,$ where $C$ is chosen such that there are the same number of infected individuals in \eqref{eq:z} at time $T=3$ as in $C \varphi_1$. As we can see in Figure \ref{fig_metapop} both the stochastic simulations and the deterministic ODE approximation are close to being proportional to the leading eigenvector as it is predicted by Theorem \ref{t:existence}.

\begin{figure}[H]
	\centerline{\includegraphics[scale=0.7]{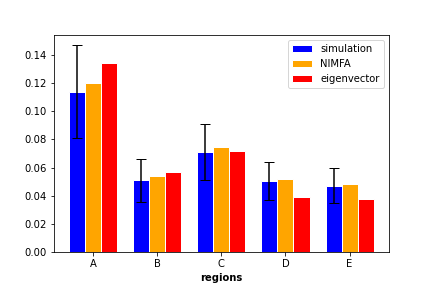}}
	\caption{Prevalence at time $T=3$ in the 5 communities in the first setting. For the simulations the bar represents the average of 30 simulations while the error bars shows $2$ times the standard deviation.   }
	\label{fig_metapop}
\end{figure} 

In the second setting we test whether two quantitatively different initial conditions would lead to similar epidemic curves at all five locations up to time translation as Theorem \ref{t:main} suggests.

We start the epidemic with a $100$ initial infections and run it until time $T=8$. The two initial conditions are:
\begin{itemize}
	\item All the infections starts from community A.
	\item Each community has $20$ initially infected individuals.
\end{itemize}

\begin{figure}[H]
	\centerline{\includegraphics[scale=0.42]{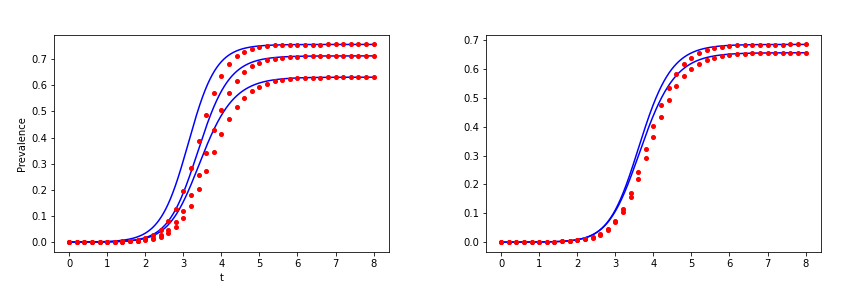}}
	\caption{Prevalence in the second setting. Left: communities A to C. Right: communities D and E. Solid lines: All initial infections start from A. Dotted lines: each community starts with $20$ infections. }
	\label{fig_USIC_metapop1}
\end{figure}

After appropriate (global) time shift the two set of curves are virtually indistinguishable.

\begin{figure}[H]
	\centerline{\includegraphics[scale=0.42]{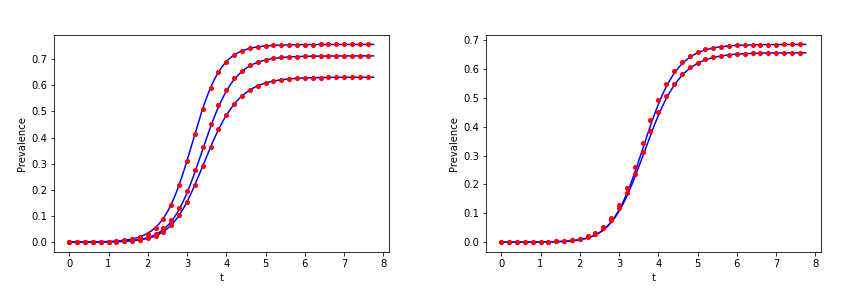}}
	\caption{Time shifted version of Figure \ref{fig_USIC_metapop1} \rem{using \eqref{eq:timeshit}} . }
	\label{fig_USIC_metapop2}
\end{figure}

\section{Proofs}
\label{s:proof}

\subsection{Proof of auxiliary statements}
\label{s:auxiliary}

Here we gathered some lemmas that are used in proofs but not vital to understanding the main ideas. 

Recall the notations from Section \ref{s:linear}.

\begin{lemma}
	\label{l:u_and_v}
	
	Denote $g(t):=\beta u(t) \W u(t).$
	\begin{align}
		\label{eq:u_and_v}
		u(t)=v(t)-\int_{t_0}^{t}e^{\mathcal{A}(t-s)}g(s) \d s
	\end{align}
\end{lemma}
\begin{corollary}
	Note that $\tilde{v}(t):=e^{\mathcal{A}(t-s)}g(s)$ is a solution to \eqref{eq:v} with initial condition $\tilde{v}(s)=g(s) \geq 0$ as $u(s) \geq 0$. The branching random walk interpretation ensures $\tilde{v}(t) \geq 0$ which along with \eqref{eq:u_and_v} implies $u(t) \leq v(t).$
\end{corollary}

\begin{proof} (Lemma \ref{l:u_and_v})
	
	Note that
	\begin{align}
		\label{eq:u_inhomogenous}
		\partial_t u&=\beta(1-u)\W u-\gamma u=(\beta \W-\gamma \mathbb{I})u-\beta u \W u=\mathcal{A}u-g.
	\end{align}
	Treating $g$ as a known function makes \eqref{eq:u_inhomogenous} an inhomogeneous linear problem with solution
	\begin{align*}
		u(t)=&e^{A(t-t_0)} \left(u(t_0)-\int_{t_0}^{t}e^{-\mathcal{A}(s-t_0)}g(s) \d s\right)=\\
		& \underbrace{e^{A(t-t_0)}u(t_0)}_{=v(t)}-\int_{t_0}^{t}e^{\mathcal{A}(t-s)}g(s) \d s.
	\end{align*}
\end{proof}

\begin{lemma}
	\label{l:c1_bound}
	Assume $\varphi_1$ is uniformly positive with constant $m>0.$ Then 
	\begin{align*}	
		\|u(t) \|_2 \leq \frac{1}{\sqrt{m}}\sqrt{c_1(t)}.
	\end{align*}
\end{lemma}
\begin{proof}(Lemma \ref{l:c1_bound})	
	\begin{align*}
		\|u(t) \|_2^2&=\int_{0}^{1}u^2(t,x) \d x \leq \int_{0}^{1}u(t,x) \d x= \frac{1}{m}\int_{0}^{1}mu(t,x) \d x\\
		&  \leq \frac{1}{m}\int_{0}^{1}\varphi_1(x)u(t,x) \d x =\frac{1}{m}c_1(t)
	\end{align*}
\end{proof}

\begin{lemma}
\label{l:bilinear}
Assume $W(x,y) \leq M$ or $W$ is rank-$1$ with $\varphi_1 \in L^{2+\rho}([0,1])$ for some $\rho>0.$ Then there are some $C>0,$ $0<\theta \leq 1$ such that
\begin{align*}
 \| u(t)\W u(t)\|_2 \leq C \| u(t)\|_2^{1+\theta}.	
\end{align*}
\end{lemma}

\begin{proof}(Lemma \ref{l:bilinear})

When $W(x,y) \leq M$ we can use
\begin{align*}
\W u(t,x)=\int_{0}^{1} W(x,y) u(t,y) \d y \leq M \int_{0}^{1} u(t,y)  \d y =M \|u(t) \|_1 ,	
\end{align*}
which implies
\begin{align*}
\|u(t) \W u(t) \|_2 \leq M \| u(t)\|_1 \|u(t) \|_2 \leq M \|u(t) \|_2^2.
\end{align*}

In the rank-$1$ case we use the Hölder inequality with $p=1+\frac{2}{\rho}, \ q=1+\frac{\rho}{2}.$ Note that $p>1$. Since $\W u(t,x)= \lambda_1\langle \varphi_1, u(t) \rangle \varphi_1(x),$
\begin{align*}
\|u(t)\W u(t) \|_2^2=&\lambda_1^2 \langle \varphi_1, u(t) \rangle^2 \| u(t) \varphi_1 \|_2^2 \leq \lambda_1^2 \underbrace{\|\varphi_1 \|_2^2}_{=1} \|u(t) \|_2^2  \langle u^2(t), \varphi_1^2 \rangle \leq \\
&\lambda_1^2 \|u(t) \|_2^2   \|u^2(t) \|_p \|\varphi_1^2 \|_q.
\end{align*} 

Note that $\|\varphi_1^2 \|_q$ is bounded as
\begin{align*}
\|\varphi_1^2 \|_q^{q}=\int_{0}^{1} \varphi_1^{2q}(x) \d x=\int_{0}^{1}\varphi_1^{2+\rho}(x) \d x<\infty.
\end{align*}
As for the $\|u^2(t) \|_p$ term
\begin{align*}
\|u^2(t) \|_p=\left(\int_{0}^{1} u^{2p}(t,x) \d x \right)^{\frac{1}{p}} \leq \left(\int_{0}^{1} u^{2}(t,x) \d x \right)^{\frac{1}{p}}= \| u(t)\|_2^{\frac{2}{p}}. 
\end{align*}
Thus,
\begin{align*}
\|u(t)\W u(t) \|_2 \leq \lambda_1 \sqrt{ \|\varphi_1^2 \|_q} \|u(t) \|_{2}^{1+\frac{1}{p}}.
\end{align*}
\end{proof}

\begin{lemma}
	\label{l:elsilon_0_error}
	
	\begin{align*}
		\sup_{0 \leq s \leq t} \|u_1(s)-u_2(s) \|_2 \leq \|u_1(0)-u_2(0) \|e^{\alpha_1 t}
	\end{align*}	
\end{lemma}

\begin{proof}(Lemma \ref{l:elsilon_0_error})
	
	Recall \eqref{eq:u^c_explicit} from the proof of Proposition \ref{t:approx}. It can be rewritten as
	\begin{align}
		\label{eq:u^c_expliciter}
		\begin{split}
			u_i^c(t)=&e^{-\gamma t} \exp \left(-\beta \int_{0}^{t} \W u_i(s) \d s\right)u_i^c(0)\\
			&+\int_{0}^{t} \gamma e^{-\gamma(t-s)} \exp \left(-\beta \int_{s}^{t} \W u_i(\tau) \d \tau \right) \d s.
		\end{split}
	\end{align}
	
	Define the error term $\Delta(t):= \|u_1(t)-u_2(t) \|_2.$ Note that $e^{-x}$ is Lipschitz with constant $1$ for $x \geq 0.$ The error arising from the first term of \eqref{eq:u^c_expliciter} can be bounded as
	\begin{align*}
		&\left \|e^{-\gamma t} \exp \left(-\beta \int_{0}^{t} \W u_1(s) \d s\right)u_1^c(0)-e^{-\gamma t} \exp \left(-\beta \int_{0}^{t} \W u_2(s) \d s\right)u_2^c(0) \right \|_2 \leq \\
		& e^{-\gamma t}\beta \lambda_1 \int_{0}^{t} \Delta (s) \d s+e^{-\gamma t} \Delta(0),
	\end{align*}
	while the second term becomes
	\begin{align*}
		&\left \| \int_{0}^{t} \gamma e^{-\gamma(t-s)} \left[  \exp \left(-\beta \int_{s}^{t} \W u_1(\tau) \d \tau \right)-\exp \left(-\beta \int_{s}^{t} \W u_2(\tau) \d \tau \right)  \right] \d s \right \|_2 \leq \\
		&\beta \lambda_1 \int_{0}^{t}\gamma e^{-\gamma(t-s)} \int_{s}^{t} \Delta(\tau) \d \tau \d s=\beta \lambda_1 \int_{0}^{t} \Delta(\tau) \int_{0}^{\tau} \gamma e^{-\gamma(t-s)} \d s \d \tau= \\
		&\beta \lambda_1 \int_{0}^{t} \left(e^{-\gamma(t-\tau) }-e^{-\gamma t} \right) \Delta(\tau) \d \tau .
	\end{align*}
	The two bounds together give
	\begin{align}
		\label{eq:Delta_error}
		\Delta(t) \leq e^{-\gamma t}\Delta(0)+\beta \lambda_1 \int_{0}^{t}e^{-\gamma(t-s)} \Delta(s) \d s.
	\end{align}
	
	Define $\tilde{\Delta}(t):=e^{\gamma t} \Delta(t)$ and multiply both sides by $e^{\gamma t}$ Grönwall's lemma concludes
	\begin{align*}
		\tilde{\Delta}(t) \leq& \tilde{\Delta}(0)+\beta \lambda_1 \int_{0}^{t} \tilde{\Delta}(s) \d s \\
		\tilde{\Delta}(t) \leq& \tilde{\Delta}(0)e^{\beta \lambda_1 t} \\ 
		\Delta(t) \leq & \Delta(0)e^{(\beta \lambda_1-\gamma)t}=\Delta(0)e^{\alpha_1 t}.
	\end{align*}
\end{proof}

\subsection{Proofs for the  existence, uniqueness and approximation of the PDE}
\label{s:PDE_proofs}

Here we give the proofs of Propositions \ref{t:PDE_basics} and \ref{t:approx}.

\begin{proof}(Proposition \ref{t:approx})
		Let $u(t)$ be a solution of \eqref{eq:u} such that $u(t) \in \Delta$ for all $t \in \mathbb{R}_{0}^{+}$. $u^c(t,x):=1-u(t,x)$ refers to the probability of being susceptible at $x$ at time $t$. It satisfies the PDE
		\begin{align}
			\label{eq:u^c}
			\partial_t u^c=-\beta u^c \W u+\gamma(1-u^c).
		\end{align}
		
		We define $L(t,x):=\beta \W u(t,x)+\gamma \geq 0$ and treat it as a known function. Solving \eqref{eq:u^c} under such assumption yields
		\begin{align}
			\nonumber
			\partial_t u^c(t)&+L(t)u^c(t)=\gamma \\
			\label{eq:u^c_explicit}
			u^c(t)=& \exp\left(-\int_{0}^{t}L(s) \d s\right)u^c(0)+\int_{0}^{t} \gamma \exp \left(-\int_{s}^{t}L(\tau) \d \tau \right) \d s.
		\end{align}
		
		Let now $u_1,u_2$ be two solutions of \eqref{eq:u} with initial conditions $u_1(0),u_2(0) \in \Delta$ and kernels $W_1,W_2$. Note that $u_1^c(t)-u_2^c(t)=-(u_1(t)-u_2(t))$ and 
		\begin{align}
		\label{eq:L_difference}
		L_1(t)-L_2(t)=\beta \left[\left(\W_1-\W_2\right)u_1(t)+\W_2\left(u_1(t)-u_2(t)\right) \right].
		\end{align}
	
		Also, note that $e^{-x}$ is Lipschitz continuous with constant $1$ for $x \geq 0$.  
		Thus, for $t \in [0,T]$ 
		\begin{align*}
			\|u_1(t)-u_2(t) \|_2 \leq& \left \| \left[\exp\left(-\int_{0}^{t}L_1(s) \d s\right)-\exp\left(-\int_{0}^{t}L_2(s) \d s\right) \right] u_1(0) \right \|_2 + \\
			&\left \|\exp\left(-\int_{0}^{t}L_2(s) \d s\right)[u_1(0)-u_2(0)] \right \|_1+\\
			& \left \|\int_{0}^{t} \gamma  \left[\exp \left(-\int_{s}^{t}L_1(\tau) \d \tau \right)-\exp \left(-\int_{s}^{t}L_2(\tau) \d \tau \right) \right] \d s \right \|_2 \\
			&\|u_1(0)-u_2(0) \|_2+(1+\gamma T)\int_{0}^{t} \|L_1(s)-L_2(s) \|_2 \d s \\
			\leq & \|u_1(0)-u_2(0) \|_2+(1+\gamma T)\beta T \|W_1 -\W_2\|_2\\
			&+(1+\gamma T)\beta \lambda\int_{0}^{t} \|u_1(s)-u_2(s) \|_2 \d s\\
			\sup_{0\leq t \leq T}\|u_1(t)-u_2(t) \|_2=& O\left(\|u_1(0)-u_2(0) \|_2+\|\W_1-\W_2 \|_2 \right).
		\end{align*}
	
	To finish, we observe $\|\W_1-\W_2 \|_2 \leq \|W_1-W_2 \|_2.$		 
\end{proof}

\begin{proof}(Proposition \ref{t:PDE_basics})
	The uniqueness simply follows from Proposition \ref{t:approx} by setting $u_1(0)=u_2(0)$ and $W_1=W_2$.

	Let $u_N$ be a sequence of solutions with kernel
	$$W_{N}(x,y):=\max\{W(x,y),N\} $$
	and initial condition $u_N(0)=u(0).$  \cite[Proposition 2.9. (ii)]{SISdyn} guarantees that such solutions exist on $\mathbb{R}_{0}^{+}$ and $u_N(t) \in \Delta.$

	From \eqref{eq:approx} we conclude
	\begin{align*}
		\sup_{0 \leq t \leq T} \| u_N(t)-u_M(t)\|_2=O(\|W_N-W_M \|_2) \to 0
	\end{align*}
	as $N,M \to \infty$, making $\left( u_N \right)_{N=1}^{\infty}$ a Cauchy-sequence on  $\mathcal{C}\left([0,T],L^2([0,1])\right).$ Since $T$ is arbitrary, the domain of $u$ can be extended to $\mathbb{R}_{0}^{+}.$
	
	It is straightforward to check that $u(t) \in \Delta$ and satisfies \eqref{eq:u}.
\end{proof}

\subsection{Proof of USIC}
\label{s:USIC_proof}

We mention that when $W$ is uniformly positive, then $\varphi_1$ is also uniformly positive as
\begin{align*}
	0<m:=\frac{m_0\| \varphi_1 \|_1}{\lambda_1} \leq \frac{1}{\lambda_1} \int_{0}^{1}W(x,y) \varphi_1(y) \d y = \varphi_1(x),
\end{align*} 
hence $\varphi_1(x)\geq m>0$ can be assumed in all three cases of Theorem \ref{t:main}.

First, we break down the proof to four claims - each with their own subsection - that together imply Theorem \ref{t:main}. \rem{These claims represent different scales which requiers different kind of arguments.
	
In Claim \ref{c:epsilon'} we establish the time translation where both of the solutions  $\approx \varepsilon' \varphi_1+o(\varepsilon')$ while the linearization is still accurate. Here $\varepsilon'$ is arbitrarily small.

Next, in Claim \ref{c:epsilon_0} we fix a small $\varepsilon_0$ and show that even when the solutions look like $\varepsilon_0 \varphi_0$ and the linearization breaks down the error is still small in terms of $\varepsilon'$.

Claim \ref{c:epsilon_close} show that the during the bulk of the process much error can not accumulate since it takes $O(1)$ amount of time.

Finally, Claim \ref{c:monotonicity} show that once we are $\varepsilon/2$ close to the stationary solution, we remain close forever.

Since in non of the phases through Claim \ref{c:epsilon_0} to Claim \ref{c:monotonicity} did we accumulate a prohibitive amount of error, we can get bellow a total error of $\varepsilon$ given that we set $\varepsilon'$ to be small enough and choosing a corresponding $\delta$.
}

We implicitly assume a), b) or c) throughout Claim \ref{c:epsilon'} to \ref{c:monotonicity}.

\begin{claim}
\label{c:epsilon'}
$\exists 0<\theta \leq 1, \ 0<\delta_0  \ \forall \epsilon',\eta>0 \ \exists 0<\delta \leq \delta_0 \ \forall u(0) \in \Delta$ (in Case a) $u(0) \in \Delta_I$) such that $0<\| u(0)\|_2 \leq \delta \ \Rightarrow \ \exists T \geq 0$ such that
\begin{align}
	\label{eq:epsilon'_1}
	u(T)=\varepsilon' \varphi_1+O \left((\varepsilon')^{1+\theta} \right) \ \textit{in } L^2([0,1])
\end{align}
while
\begin{align}
	\label{eq:epsilon'_2}
	\sup_{0 \leq t \leq T} \|u(t) \|_2 \leq \eta.
\end{align}  	
\end{claim}

In Claim \ref{c:epsilon'} $\eta$ and $\delta$ are the same as in Theorem \ref{t:main}, however, $\varepsilon'$ is dummy variable that plays a similar role to $\varepsilon$. Roughly speaking Claim \ref{c:epsilon'} states that after fixing $\varepsilon',\eta$ a solution $u$ with small enough initial condition will look like $u(T)=\varepsilon' \varphi_1+o(\varepsilon')$ for some $T$ while being smaller than $\eta$ on $[0,T]$. 

Next, we have two solutions $u_1,u_2$ where the initial conditions are the endpoints of Claim \ref{c:epsilon'}.

\begin{claim}
\label{c:epsilon_0}
Let $u_1,u_2$ two solutions such that
\begin{align}
\label{eq:epsilon_0_1}
u_1(0),u_2(0)=\varepsilon' \varphi_1+O \left((\varepsilon')^{1+\theta} \right) \ \textit{in } L^2([0,1])
\end{align}
for some $0<\theta<1$ Let $0<\varepsilon' \leq \varepsilon_0<1$. Then for time $t^*:=\frac{1}{\alpha_1} \log \frac{\varepsilon_0}{\varepsilon'}$
\begin{align}
\label{eq:epsilon_0_2}
u_1(t^*),u_2(t^*)=\varepsilon_0 \varphi_1+O \left((\varepsilon')^{1+\theta}\right) \in \textit{in } L^2([0,1])
\end{align}
while
\begin{align*}
	 \sup_{0 \leq t \leq t^*} \|u_1(t)-u_2(t) \|_2=O \left((\varepsilon')^\theta \right).
\end{align*}
\end{claim}

In Claim \ref{c:epsilon_0} we think about $\varepsilon_0$ as a small, but fixed (non vanishing) number that only depends on $W,\beta$ and $\gamma$. At the level of $\varepsilon_0$ the error between $u$ and $\varphi_1$ is small, but non vanishing, however, the error between $u_1$ and $u_2$ is $o(1)$.

After we reached the level $\varepsilon_0$ the renaming time to get $\varepsilon$ close to the endemic state $\psi$ will no longer depend on $\varepsilon'$ but merely on $\varepsilon$, justifying the step.

\begin{claim}
\label{c:epsilon_close}
Let $u_1,u_2$ be two solutions such that
\begin{align*}
	u_1(t^*),u_2(t^*)=&\varepsilon_0 \varphi_1+O \left((\varepsilon_0)^{1+\theta}\right) \in \textit{in } L^2([0,1]) \\
	\sup_{0 \leq t \leq t^*} \|u_1(t)-u_2(t) \|_2=&O \left((\varepsilon')^\theta \right).
\end{align*}
for some $0<\theta, \ 0<\varepsilon' \leq  \varepsilon_0 <1$.

Then if $\epsilon_0$ is small enough (depending only on $W,\beta$ and $\gamma$) one has $\forall 0<\varepsilon' \leq \varepsilon<1 \ \exists t^{**}=t^{**}(\varepsilon)$ such that
\begin{align}
\label{eq:epsilon_1}
(i=1,2) \ \|u_i(t^{**})-\psi \|_2 \leq \frac{\varepsilon}{2}
\end{align}
if $\gamma=0$
\begin{align}
	\label{eq:epsilon_1'}
	(i=1,2) \ \|u_i(t^{**})-\psi \|_\pi \leq \frac{\varepsilon}{2}
\end{align}
 if $\gamma>0$, while
\begin{align}
\label{eq:epsilon_2}
\sup_{0 \leq t \leq t^{**}} \|u_1(t)-u_2(t) \|_2=O \left((\varepsilon')^{\theta}e^{\alpha_1 t^{**}(\varepsilon)} \right).	
\end{align}
\end{claim}

The message of Claim \ref{c:epsilon_close} is that after fixing $\varepsilon$ we only constant amount of time to get from the level $\varepsilon_0$ to being $\varepsilon$ close to the endemic equilibrium $\psi$, hence, the accumulated error in this segment is still vanishing.

\begin{claim}
\label{c:monotonicity}
$\|u(t)-\psi  \|_\pi$ is monotone decreasing when $\gamma>0$ and $\|u(t)-\psi  \|_2$ is monotone decreasing when $\gamma=0$.
\end{claim}

\begin{corollary}
\label{corr:monotone}
	 Assume $\gamma>0$. If $u_1,u_2$ are two solutions such that
	\begin{align*}
 (i=1,2) \ \|u_i(0)-\psi \|_{\pi} \leq \frac{\varepsilon}{2}		
	\end{align*}
	then
	\begin{align*}
	\sup_{t \geq 0} \|u_1(t)-u_2(t) \|_2 \leq&  \sup_{t \geq 0} \|u_1(t)-\psi \|_\pi+\sup_{t \geq 0} \|u_2(t)-\psi \|_\pi= \\
	&\|u_1(0)-\psi \|_\pi+\|u_2(0)-\psi \|_\pi \leq \varepsilon. 
	\end{align*}
	
	Similarly, when $\gamma=0$ and
	\begin{align*}
		(i=1,2) \ \|u_i(0)-\psi \|_{2} \leq \frac{\varepsilon}{2}		
	\end{align*}
	then
	\begin{align*}
		\sup_{t \geq 0} \|u_1(t)-u_2(t) \|_2 \leq&  \sup_{t \geq 0} \|u_1(t)-\psi \|_2+\sup_{t \geq 0} \|u_2(t)-\psi \|_2= \\
		&\|u_1(0)-\psi \|_2+\|u_2(0)-\psi \|_2 \leq \varepsilon. 
	\end{align*}
\end{corollary}

\begin{figure}[t]
	\centerline{\includegraphics[scale=0.6]{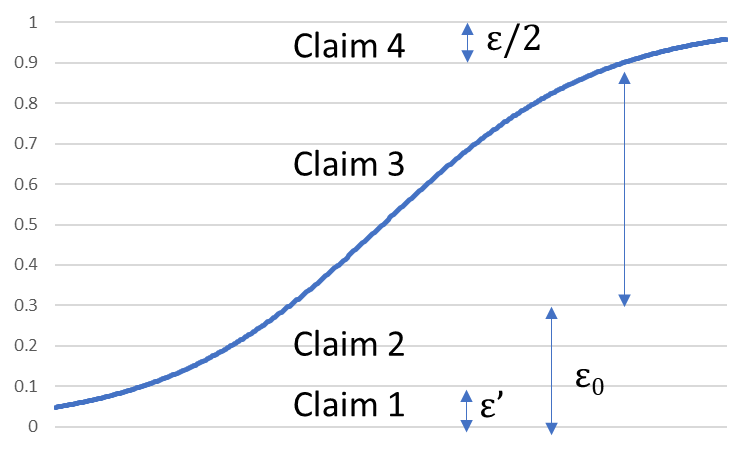}}
	\caption{A schematic representation of the scales on which Claim \ref{c:epsilon'} to \ref{c:monotonicity} operates. }
	\label{fig_claim}
\end{figure} 

\begin{proof} (Theorem \ref{t:main})

First we set $\delta_0$ in Claim \ref{c:epsilon'} and $\varepsilon_0$ in Claim \ref{c:epsilon_close}. They are constants that merely depend on $W,\beta$ and $\gamma$.

Next we choose $0<\varepsilon,\eta < 1$ to be arbitrarily small. Then, we set $\varepsilon'$ to be so small such that $O \left((\varepsilon')^{\theta}e^{\alpha_1 t^{**}(\varepsilon)} \right) \leq \varepsilon$ and $\varepsilon' \leq \varepsilon,\varepsilon_0$. Lastly, we choose $\delta \leq  \delta_0$ such that it is small enough for Claim \ref{c:epsilon'}.

For two solutions $u_1,u_2$ with $0<\|u_1(0) \|_2 \leq \delta$ we choose $t_1,t_2$ to be the appropriate $T$ in Claim \ref{c:epsilon'} respectively, hence $\|u_1(t_1)-u_2(t_2) \|_2 \leq \varepsilon$ while ($i=1,2$) $\sup_{0 \leq t \leq t_i} \|u_i(t) \|_2 \leq \eta$.
Claim \ref{c:epsilon_0}, \ref{c:epsilon_close} and Corollary \ref{corr:monotone} ensures that on the segments $[0,t^*], \ [t^*,t^{**}]$ and $[t^{**}, \infty[ $ the error $\|u_1(t_1+t)-u_2(t_2+t) \|_2$ remains below $\varepsilon$, concluding the proof. 
\end{proof}

\subsubsection{Reaching the level $\varepsilon'$}
\label{s:varepsilon'}
   
We fix some small $\varepsilon' , \eta>0$ to which we provide an propitiate $\delta>0$ controlling the size of the initial conditions. More precisely, we assume
\begin{align*}
0<  \|u(0) \|_2 \leq \delta.
\end{align*}

In this section $t_0$ will either be $0$ or it takes an other, positive value specified later.

Define the time it takes for the leading term in \eqref{eq:v_expanded} to reach the level $\varepsilon'$ from $c_1(t_0)$ as 
\begin{align}
\label{def:t_bar}
\bar{t}&:= \frac{1}{\alpha_1}\log \frac{\varepsilon'}{c(t_0)},
\end{align}
which is equivalent to $\tilde{c}_1(t_0+\bar{t})=\varepsilon'.$   

There are two crucial steps to make our heuristic rigorous until we reach the $\varepsilon'$-level: we should mitigate the error arising from linearization, and prove that the leading term in \eqref{eq:v_expanded} is indeed dominant at time $t_0+\bar{t}.$  The following two lemmas work towards these goals.
\begin{lemma}
\label{l:linear_error}
\begin{align}
\label{eq:linear_error}
\sup_{t_0 \leq t \leq t_0+\bar{t} } \|u(t)-v(t) \|_2 \leq \frac{\beta C}{\theta\alpha_1} \left(\frac{\|u(t_0) \|_2}{c_1(t_0)} \varepsilon'\right)^{1+\theta},
\end{align}
where $C,\theta>0$ are constants from Lemma \ref{l:bilinear}.
\end{lemma}

\begin{proof}(Lemma \ref{l:linear_error})
	
	Recall Lemma \ref{l:u_and_v} and \ref{l:bilinear}.
	\begin{align}
		\nonumber
		\|g(s) \|_2&=\beta \|u(s) \W u(s) \|_2 \leq \beta C \| u(s)\|_2^{1+\theta} \overset{u \leq v}{\leq} \beta C \| v(s)\|_2^{1+\theta} \\
		\nonumber
		&\overset{\eqref{eq:v_expanded}}{\leq} \beta C \| u(t_0)\|_2^{1+\theta}e^{ (1+\theta)\alpha_1(s-t_0)}\\
		\nonumber
		\|u(t)-v(t) \|_2 &\leq \int_{t_0}^{t} \left\|e^{\mathcal{A}(t-s)}g(s)  \right\|_2 \d s \leq \int_{t_0}^{t}e^{\alpha_1(t-s)} \left\|g(s)  \right\|_2 \d s \\
		\nonumber
		& \leq \beta C \| u(t_0)\|_2^{1+\theta}e^{\alpha_1(t-t_0)}\int_{t_0}^{t}e^{\theta \alpha_1(s-t_0)}  \d s \\
		\label{eq:linear_error2}
		& \leq \frac{\beta C}{\theta \alpha_1} \| u(t_0)\|_{2}^{1+\theta}e^{(1+\theta)\alpha_1(t-t_0)}
	\end{align}

	Thus,
	\begin{align*}
		\sup_{t_0 \leq t \leq t_0+\bar{t}}\|u(t)-v(t) \|_2 \leq \frac{\beta C}{\theta \alpha_1} \left(\| u(t_0)\|_2e^{\alpha_1\bar{t}} \right)^{1+\theta}=\frac{\beta C}{\theta \alpha_1} \left(\frac{\| u(t_0)\|_2}{c_1(t_0)} \varepsilon' \right)^{1+\theta}.
	\end{align*}	
\end{proof}

\begin{lemma}
\label{l:leading_term}

\begin{align}
\label{eq:leading_term}
v(t_0+ \bar{t})=\varepsilon' \left[\varphi_1+O  \left(\frac{\|u(t_0) \|_2}{c_1(t_0)}  e^{-(\alpha_1-\alpha_2) \bar{t}}\right) \right] \ \ \textit{in } L^2([0,1]).
\end{align}
\end{lemma}

\begin{proof}(Lemma \ref{l:leading_term})
	
	\begin{align*}
		v(t_0+\bar{t}) &\overset{\eqref{eq:v_expanded}}{=}c_1(t_0)e^{\alpha_1\bar{t}}\varphi_1+\sum_{k>1}c_k(t_0)e^{\alpha_k \bar{t}}\varphi_k\\
		&=\underbrace{c_1(t_0)e^{\alpha_1\bar{t}}}_{=\varepsilon'}\left(\varphi_1+\frac{1}{c_1(t_0)}\sum_{k>1}c_k(t_0)e^{-(\alpha_1-\alpha_k) \bar{t}}\varphi_k \right)
	\end{align*}	
	
	\begin{align*}
		&\left \|\frac{1}{c_1(t_0)}\sum_{k>1}c_k(t_0)e^{-(\alpha_1-\alpha_k) \bar{t}}\varphi_k \right \|_2^2  = \frac{1}{c_1^2(t_0)}\sum_{k>1}c_k^2(t_0)e^{-2(\alpha_1-\alpha_k) \bar{t}}\leq \\
		& \frac{e^{-2(\alpha_1-\alpha_2) \bar{t}}}{c_1^2(t_0)}\sum_{k>1}c_k^2(t_0) \leq \left(\frac{\| u(t_0)\|_2}{c_1(t_0)} \right)^2e^{-2(\alpha_1-\alpha_2) \bar{t}}
	\end{align*}
\end{proof}

\begin{remark}
\label{r:counter_example}

The problem with Lemma \ref{l:linear_error} and \ref{l:leading_term} is the appearance of the expression $\frac{\|u(t_0) \|_2}{c_1(t_0)}$ in the error terms which unfortunately can be arbitrarily large. One such example is the homogeneous population $W(x,y) \equiv 1$ with 
\begin{align}
\label{eq:intial_conentration}
u(t_0,x)=\delta \1{0 \leq x \leq \delta};
\end{align}
then
\begin{align*}
c_1(t_0)&=\int_{0}^{1} u(t_0,x)\underbrace{\varphi_1(x)}_{=1} \d x=\delta^2 \\
\|u(t_0) \|_2^2&=\int_{0}^{1}u^2(t_0,x) \d x= \delta^3 \\
\frac{\|u(t_0) \|_2}{c_1(t_0)}&=\delta^{-\frac{1}{2}} \to \infty \ \textit{as } \delta \to 0^+.
\end{align*} 

Note that
$$\|u(t_0) \|_2^2=\sum_{k=1}^{\infty}c_k^2(t_0), $$
hence, $\frac{c_1(t_0)}{\|u(t_0) \|_2}$ measures how much weight the $\varphi_1$ component has initially, resulting in large error terms when being small.
\end{remark}

There are two ways to deal with this problem.

First is setting $t_0=0$ to be the beginning and only considering small initial conditions where $\varphi_1$ already has enough weight, or, in other words, assuming $\frac{\| u(0)\|_2}{c_1(0)} \leq K$ for some constant $K$. This limits the class of initial conditions we may consider.

This approach works well for the construction of eternal solutions where the initial condition is chosen to be roughly $\varepsilon \varphi_1$, making the $\varphi_1$ component dominant from the beginning (see the proof of Theorem \ref{t:existence} in Section \ref{s:proof_eternal}).

Another important case when such an assumption holds naturally is when $W$ is discrete, as we are only considering piece-wise constant initial conditions from $\Delta_I$ artificially excluding counterexamples like in Remark \ref{r:counter_example}.

\begin{lemma}
\label{l:discrite_initial_bound}
Assume $W$ discrete with $J:=\min_{i} |I_i|$ and $u(0) \in \Delta_I \setminus \{0\}.$ Then
\begin{align*}
\left(\frac{\|u(0) \|_2}{c_1(0)} \right)^2 \leq \frac{1}{m^2 J}.
\end{align*}
\begin{proof}(Lemma \ref{l:discrite_initial_bound})
	
	\begin{align*}
		\left(\frac{\| u(0)\|_2}{c_1(0)} \right)^2&=\frac{\int_{0}^{1}u^2(0,x) \d x}{\left( \int_{0}^{1} \varphi_1(x) u(0,x) \d x \right)^2} \leq \frac{1}{m^2} \int_{0}^{1} \left( \frac{u(0,x)}{\|u(0) \|_1} \right)^2 \d x\\
		&=:\frac{1}{m^2} \int_{0}^{1} f^2(x) \d x \leq \frac{\| f\|_{\infty}}{m^2}
	\end{align*}	
	as $f(x):=\frac{u(0,x)}{\|u(t_0) \|_1}$ is a density function.
	
	It remains to give an upper bound on $f(x).$
	\begin{align*}
		(x \in I_i) \ \ f(x)= \frac{z_i(0)}{\sum_{j=1}^n z_j(0)|I_j|} \leq \frac{1}{J} \frac{z_i(0)}{\sum_{j=1}^n z_j(0)} \leq \frac{1}{J}.
	\end{align*}
\end{proof}

\end{lemma}

\begin{corollary}
Assuming $\left(\frac{\|u(0) \|_2}{c_1(0)} \right)^2$ is bounded at $t_0=0$, Lemmas \ref{l:linear_error} and \ref{l:leading_term} lead to
\begin{align}
	\label{eq:epsilon'_level1}
	u(\bar{t})=\varepsilon' \left[\varphi_1+O\left((\varepsilon')^{\theta}+e^{-(\alpha_1-\alpha_2)\bar{t}} \right) \right] \ \ \textit{in } L^2([0,1]).
\end{align}
It is worth noting that
$$\bar{t}=\frac{1}{\alpha_1}\log \frac{\varepsilon'}{c(0)} \to  \infty \ \ \textit{as } \delta \to 0^{+}, $$ 	
since $c_1(0) \leq \|u(0) \|_2 \leq \delta.$ This will be relevant for Corollary \ref{cor:t_bar}.
\end{corollary}

The other method is to initiate a new period and set $t_0=\hat{t},$ where $\hat{t}$ is the time till $\varphi_1$ receives enough weight so that $\frac{\| u(\hat{t})\|_2}{c_1(\hat{t})}=O(1).$ The potential danger with this approach is that this event might happen \emph{later} than when we reach the level $\varepsilon'$.

\begin{heuristics}
With the choice of
\begin{align}
	\label{eq:t_hat}
	\hat{t}:= \frac{1}{\alpha_1-\alpha_2}\log \frac{\| u(0)\|_2}{c_1(0)}
\end{align}
we can guarantee
\begin{align*}
	\left(\frac{\|v(\hat{t}) \|_2}{\tilde{c}_1(\hat{t})} \right)^2&=1+\frac{1}{c_1^2(0)}\sum_{k>1}c_k^2(0)e^{-2(\alpha_1-\alpha_k)\hat{t}} \\
	& \leq 1+\left( \frac{\|u(0) \|_2}{c_1(0)}\right)^2e^{-2(\alpha_1-\alpha_2)\hat{t}}=2
\end{align*}
where $v$ here stands for the solution of \eqref{eq:v} with initial condition $v(0)=u(0).$ 
\end{heuristics}

However, what we need to bound is $\frac{\|u(\hat{t}) \|_2}{c_1(\hat{t})}$ instead, which can be achieved via the following lemma:
\begin{lemma}
\label{l:M_t_hat}
\begin{align}
\label{eq:M_t_hat}
\left(\frac{\|u(\hat{t}) \|_2}{c_1(\hat{t})} \right)^2 \leq \frac{2}{\left(1-\frac{\beta}{c_1(0)}\int_{0}^{\hat{t}}e^{-\alpha_1 s} \langle u(s)\W u(s), \varphi_1 \rangle \d s \right)^2}
\end{align}
\end{lemma}

\begin{proof}(Lemma \ref{l:M_t_hat})

Recall Lemma \ref{l:u_and_v}. Let $v$ be the solution of \eqref{eq:v} with $v(0)=u(0).$

\begin{align*}
	u(t)&=v(t)-\beta\int_{0}^{t}e^{\mathcal{A}(t-s)}u(s)\W u(s) \d s\\
	&=v(t)-\beta\int_{0}^{t}\sum_{k=1}^{\infty}\langle u(s)\W u(s), \varphi_k \rangle \underbrace{e^{\mathcal{A}(t-s)} \varphi_k}_{e^{\alpha_k(t-s)}\varphi_k}  \d s\\
	c_1(t)&=c_1(0)e^{\alpha_1 t}-\beta\int_{0}^{t}e^{\alpha_1(t-s)}\langle u(s)\W u(s), \varphi_1 \rangle \d s \\
	&=c_1(0)e^{\alpha_1 t} \left(1-\frac{\beta}{c_1(0)}\int_{0}^{t}e^{-\alpha_1 s}\langle u(s)\W u(s), \varphi_1 \rangle \d s \right)
\end{align*}

$\|u(t) \|_2$ will be bounded in the usual way.
\begin{align*}
	\| u(t)\|_2^2 & \leq \| v(t)\|_2^2 \leq c_1^2(0)e^{2\alpha_1 t}+\|u(0) \|_2^2e^{2 \alpha_2 t}
\end{align*}

This results in
\begin{align*}
	\left(\frac{\|u(\hat{t}) \|_2}{c_1(\hat{t})}\right)^2 &\leq \frac{1+\left(\frac{\| u(0)\|_2}{c_1(0)}\right)^2e^{-2(\alpha_1-\alpha_2)\hat{t}}}{\left(1-\frac{\beta}{c_1(0)}\int_{0}^{\hat{t}}e^{-\alpha_1 s}\langle u(s)\W u(s), \varphi_1 \rangle \d s \right)^2} \\
	& \overset{\eqref{eq:t_hat}}{=} \frac{2}{\left(1-\frac{\beta}{c_1(0)}\int_{0}^{\hat{t}}e^{-\alpha_1 s}\langle u(s)\W u(s), \varphi_1 \rangle \d s \right)^2}. 
\end{align*}
\end{proof}

Thus, it remains to bound the denominator of \eqref{eq:M_t_hat}. We will do so separately for assumption b) and c) in the following two lemmas.

\begin{lemma}
\label{l:rank_1}	
	
Assume $W$ is rank-$1$, $\varphi_1(x) \geq m>0,$ and $\varphi_1 \in L^{2+\rho}([0,1])$ for some $\rho>0$.

Then there is a $\delta_0$ (depending only on $W,\beta$ and $\gamma$) such that for any $0<\delta \leq \delta_0$ we have
$$\frac{\beta}{c_1(0)}\int_{0}^{\hat{t}}e^{-\alpha_1 s} \langle u(s)\W u(s), \varphi_1 \rangle \d s \leq \frac{1}{2}. $$
\end{lemma}

\begin{proof}(Lemma \ref{l:rank_1})
	
	Note that $\W u(s)=\lambda_1 \varphi_1 c_1(s) \leq  \lambda_1 \varphi_1 c_1(0)e^{\alpha_1 s}.$ We have
	\begin{align*}
		\frac{\beta }{c_1(0)}\int_{0}^{\hat{t}}e^{-\alpha_1 s} \langle u(s) \W u(s), \varphi_1 \rangle \d s \leq  \beta \lambda_1 \int_{0}^{\hat{t}} \left \langle u(s) , \varphi_1^2 \right \rangle \d s
	\end{align*}
	
	Use Hölder's inequality with $p=1+\frac{2}{\rho}, \ q=1+\frac{\rho}{2}.$
	\begin{align*}
		\left \langle u(s) , \varphi_1^2 \right \rangle \leq \|u(s) \|_p \left \|\varphi_1^2 \right \|_q.
	\end{align*}
	Since $\varphi_1 \in L^{2+\rho}([0,1])$
	\begin{align*}
		\left \|\varphi_1^2 \right \|_q^q=\int_{0}^{1} \varphi_1^{2q}(x) \d x=\int_{0}^{1} \varphi_1^{2+\rho}(x) \d x<\infty.
	\end{align*}
	As for the first term, since $0 \leq u(s,x) \leq 1$ and $1<p<\infty$ we have
	\begin{align*}
		\|u(s) \|_p^p=&\int_{0}^{1}\left(u(s,x) \right)^p \d x \leq \int_{0}^{1}u(s,x) \d x \leq \frac{1}{m} \int_{0}^{1}\varphi_1(x)u(s,x) \d x\\
		=&\frac{1}{m}c_1(s) \leq \frac{1}{m}c_1(0)e^{\alpha_1 s}
	\end{align*}
	leading to
	\begin{align*}
		\int_{0}^{\hat{t}} \left \langle u(s) , \varphi_1^2 \right \rangle \d s \leq \left(\frac{c_1(0)}{m} \right)^{\frac{1}{p}}\left \| \varphi_1^2 \right \|_q \int_{0}^{\hat{t}} e^{\frac{\alpha_1}{p}s} \d s \leq \frac{p}{\alpha_1}m^{-\frac{1}{p}} \left \| \varphi_1^2 \right \|_q \left(c_1(0)e^{\alpha_1 \hat{t}} \right)^{\frac{1}{p}}.
	\end{align*}
	
	Based on Lemma \ref{l:c1_bound},
	\begin{align}
		\label{eq:t_hat_exponent}
		\begin{split}
			e^{\alpha_1 \hat{t}}=& \left(\frac{\| u(0)\|_2}{c_1(0)}\right)^{\frac{\alpha_1}{\alpha_1-\alpha_2}}=O \left(c_1(0)^{-\frac{\alpha_1}{2(\alpha_1-\alpha_2)}}\right)\\
			c_1(0)e^{\alpha_1 \hat{t}}=&O \left(c_1(0)^{1-\frac{\alpha_1}{2(\alpha_1-\alpha_2)}}\right) \to 0
		\end{split}
	\end{align}
	as $\delta \to 0^{+}$ since
	\begin{align*}
		\frac{\alpha_1}{2(\alpha_1-\alpha_2)}=\frac{\beta \lambda_1-\gamma}{2\beta\lambda_1}=\frac{1}{2} \left( 1-\frac{\gamma}{\beta \lambda_1}\right)<1.
	\end{align*} 
	Hence, we can find a small enough $\delta_0$ such that $0<\delta \leq \delta_0$ implies
	$$\frac{\beta}{c_1(0)}\int_{0}^{\hat{t}}e^{-\alpha_1 s} \langle u(s) \W u(s), \varphi_1 \rangle \d s \leq \frac{1}{2}. $$
\end{proof}

\begin{lemma}
\label{l:bounded_W}
Assume $W(x,y) \leq M,$ $\varphi_1(x) \geq m>0$  and $\beta \lambda_1 < \gamma+2 \beta(\lambda_1-\lambda_2).$ Then there is a $\delta_0$ (depending only on $W,\beta$ and $\gamma$) such that 
for any $0<\delta \leq \delta_0$ we have
$$\frac{\beta}{c_1(0)}\int_{0}^{\hat{t}}e^{-\alpha_1 s} \langle u(s)\W u(s), \varphi_1 \rangle \d s \leq \frac{1}{2}.$$
\end{lemma}
\begin{proof} (Lemma \ref{l:bounded_W})
	Since $W$ is bounded and $\varphi_1(x) \geq m$,
	\begin{align}
		\label{eq:Wu_upper_bound}
		\begin{split}
			\W u(t,x)=&\int_{0}^{1} W(x,y)u(t,y) \d y  \leq \frac{M}{m} \int_{0}^{1} m u(t,y) \d y \leq \\
			& \frac{M}{m} \int_{0}^{1} \varphi_1(y) u(t,y) \d y= \frac{M}{m} \langle u(t),\varphi_1 \rangle=\frac{M}{m} c_1(t)
		\end{split}
	\end{align}
	uniformly in $x \in [0,1].$
	
	\begin{align*}
		&\frac{\beta}{c_1(0)}\int_{0}^{\hat{t}}e^{-\alpha_1 s} \langle u(s) \W u(s), \varphi_1 \rangle \d s \leq \frac{\beta M}{mc_1(0) }\int_{0}^{\hat{t}}e^{-\alpha_1 s}c_1(s) \langle u(s) , \varphi_1 \rangle \d s=\\
		&\frac{\beta M}{mc_1(0) }\int_{0}^{\hat{t}}e^{-\alpha_1 s}c_1^2(s) \d s \leq \frac{\beta M}{mc_1(0) }\int_{0}^{\hat{t}}e^{-\alpha_1 s}c_1^2(0)e^{ 2 \alpha_1 s} \d s=\\
		& \frac{\beta Mc_1(0)}{m }\int_{0}^{\hat{t}}e^{  \alpha_1 s} \d s \leq \frac{\beta Mc_1(0)}{ \alpha_1 m }e^{\alpha_1 \hat{t}}
	\end{align*}
	
	Using \eqref{eq:t_hat_exponent}
	\begin{align*}
		c_1(0)e^{\alpha_1 \hat{t}}=O \left(c_1(0)^{1-\frac{\alpha_1}{2(\alpha_1-\alpha_2)}}\right) \to 0
	\end{align*}
	as $\delta \to 0^{+}$ since $\frac{\alpha_1}{2(\alpha_1-\alpha_2)}<1$ due to our assumption. Hence, we can find a small enough $\delta_0$ such that $0<\delta \leq \delta_0$ implies
	$$\frac{\beta}{c_1(0)}\int_{0}^{\hat{t}}e^{-\alpha_1 s} \langle u(s) \W u(s), \varphi_1 \rangle \d s \leq \frac{1}{2}. $$
\end{proof}

\begin{corollary}
Assuming the denominator of \eqref{eq:M_t_hat} is bounded from below, Lemmas \ref{l:linear_error}, \ref{l:leading_term} and \ref{l:M_t_hat} result in
\begin{align}
	\label{eq:epsilon'_level2}
	u(\hat{t}+\bar{t})=\varepsilon' \left[\varphi_1+O\left((\varepsilon')^{\theta}+e^{-(\alpha_1-\alpha_2)\bar{t}} \right) \right] \ \ \textit{in } L^2([0,1]).
\end{align}
\end{corollary}

Now we show that $\bar{t}$ can be arbitrarily large as $\delta$ decreases, that is, $\varphi_1$ starts to dominate before we reach the level $\varepsilon'$.
\begin{lemma}
\label{l:domination_time}

Assume $\beta \lambda_1<\gamma+2\beta(\lambda_1-\lambda_2).$ Then
$ \bar{t} \to \infty $ as $\delta \to 0^{+}.$
\end{lemma}

\begin{remark}
The conditions of Lemma \ref{l:domination_time} can be interpreted the following way: when the spectral gap is large, the coefficient of $\varphi_1$ can increase more rapidly than the coefficients of the other components, enabling
 $\varphi_1$ to dominate before reaching the level $\varepsilon'$.

For rank-$1$ graphons, $\lambda_2=0,$ hence, the condition of Lemma \ref{l:domination_time} trivially holds under assumption c). In assumption b) said condition is explicitly required.
\end{remark}

\begin{proof}(Lemma \ref{l:domination_time})
	\begin{align*}
		\bar{t}=\frac{1}{\alpha_1} \log \frac{\varepsilon'}{c_1(\hat{t})}=\frac{1}{\alpha_1} \log \varepsilon' +\frac{1}{\alpha_1} \log \frac{1}{c_1(\hat{t})}
	\end{align*}
	Since $\varepsilon'$ is fixed, $\frac{1}{\alpha_1} \log \varepsilon'$ is just a constant and we can neglect it.	
	
	\begin{align*}
		&\frac{1}{\alpha_1} \log \frac{1}{c_1(\hat{t})} \geq \frac{1}{\alpha_1} \log \frac{1}{c_1(0)e^{\alpha_1\hat{t}}}=\frac{1}{\alpha_1} \log \frac{1}{c_1(0)}-\hat{t}=\\
		&-\frac{1}{\alpha_1} \log c_1(0)-\frac{1}{\alpha_1-\alpha_2}\log \frac{\| u(0)\|_2}{c_1(0)}=-\log\left[ c_1(0)^{\frac{1}{\alpha_1}} \left(\frac{\|u(0) \|_2}{c_1(0)}\right)^{\frac{1}{\alpha_1-\alpha_2}} \right]
	\end{align*}
	
	Clearly, it is enough to show that the argument of the logarithm is small. Using Lemma \ref{l:c1_bound} with $\nu:=\frac{1}{\alpha_1}-\frac{1}{2(\alpha_1-\alpha_2)}$ shows
	\begin{align*}
		c_1(0)^{\frac{1}{\alpha_1}} \left(\frac{\|u(0) \|_2}{c_1(0)}\right)^{\frac{1}{\alpha_1-\alpha_2}} = O \left( c_1^{\nu}(0) \right).
	\end{align*}
	
	Observe
	\begin{align*}
		\nu >& 0 \\
		\frac{1}{\alpha_1} >& \frac{1}{2(\alpha_1-\alpha_2)} \\
		\alpha_1 <& 2(\alpha_1-\alpha_2) \\
		\beta \lambda_1-\gamma<&2\beta(\lambda_1-\lambda_2) 
	\end{align*}
	making $\nu$ positive by the assumption.
	
	Therefore $\bar{t} \geq -\nu \log c_1(0)+O(1) \to \infty$ as $\delta \to 0^+.$ 
\end{proof}

\begin{corollary}
\label{cor:t_bar}
As $\bar{t} \to \infty$, we may choose $\delta$ to be small enough so that $e^{-(\alpha_1-\alpha_2)\bar{t}} \leq (\varepsilon')^{\theta}$, hence, by \eqref{eq:epsilon'_level1} or \eqref{eq:epsilon'_level2},
\begin{align*}
	u(T)=\varepsilon' \varphi_1+O\left(\left(\varepsilon' \right)^{1+\theta} \right)  \ \ \textit{in } L^2([0,1]) 
\end{align*}  
where  $T=t_0+\bar{t}.$
\end{corollary}

This concludes \eqref{eq:epsilon'_1} of Claim \ref{c:epsilon'}. We are left with showing \eqref{eq:epsilon'_2}.

\begin{align*}
\sup_{t_0 \leq t \leq t_0+\bar{t}}\|u(t) \|_2 \leq \sup_{t_0 \leq t \leq t_0+ \bar{t}} \|v(t) \|_2 \leq \|u(t_0) \|_2e^{\alpha_1 \bar{t}}=\frac{\|u(t_0) \|_2}{c_1(t_0)}\varepsilon' \leq \eta
\end{align*}
when $\varepsilon'$ is small enough based on Lemma \ref{l:discrite_initial_bound} - \ref{l:bounded_W}. Under assumption a) when $t_0=0$   \eqref{eq:epsilon'_2} is already satisfied, otherwise, one must check the interval $[0,\hat{t}]$ as well. That is only needed when we assume b) or c) in which case $\gamma<\lambda_1 \beta<\gamma +2\beta(\lambda_1-\lambda_2)$ holds.

Let $v$ be the solution of \eqref{eq:v} with initial condition $v(0)=u(0).$ Using \eqref{eq:t_hat_exponent},
\begin{align}
\nonumber
\|u(t) \|_2^2 \leq& \|v(t) \|_2^2=c_1^2(0)e^{2\alpha_1 t}+\sum_{k>1}c_k^2(0)e^{\alpha_k t} \leq c_1(0)e^{2\alpha_{1}\hat{t}}+\|u(0)\|_2^2 e^{2\alpha_2 \hat{t}} \\
\nonumber
=&c_1^2(0)e^{2\alpha_1 \hat{t}} \left(1+\underbrace{ \left(\frac{\|u(0) \|_2}{c_1(0)}e^{-(\alpha_1-\alpha_2)\hat{t}} \right)^2}_{=1} \right)\\
\label{eta_bound}
\sup_{0 \leq t \leq \hat{t}}\|u(t) \|_2 \leq & \sqrt{2} c_1(0)e^{\alpha_1\hat{t}}=O \left(c_1(0)^{1-\frac{\alpha_1}{2(\alpha_1-\alpha_2)}}\right) \to 0
\end{align} 
as $\delta \to 0+,$ thus, it is smaller than $\eta$ for small enough $\delta>0.$

Thus, we showed \eqref{eq:epsilon'_2} concluding the proof of Claim \ref{c:epsilon'}.

\subsubsection{Reaching the level $\varepsilon_0$}
\label{s:epsilon_0}

In this section we want to reach a level $\varepsilon_0$ where $\varepsilon_0$ is still relatively small, but  does not depend $\varepsilon'$.

Applying an appropriate time translation, we can set
$$
u_1(0), u_2(0)=\varepsilon'\varphi_1+O\left( (\varepsilon')^{1+\theta}\right).
$$

When $\varepsilon_0$ is sufficiently small, $u_1$ and $u_2$ can still be approximated by the linearized versions $v_1,v_2$, although, since $\varepsilon_0$ is set at a constant value (depending only on $W,\beta,\gamma$), a macroscopic error term remains.

\begin{heuristics}
\label{h:error_prop}
Starting from $v(0)=\varepsilon' \varphi_1$ we need
\begin{align}
	\label{eq:t*}
	t^*:=\frac{1}{\alpha_1}\log \frac{\varepsilon_0}{\varepsilon'}
\end{align}
time to reach level $\varepsilon_0$. Also, in the linear system the error propagation rate is at most $\alpha_1$ hence
\begin{align*}
	\sup_{0 \leq t \leq t^*}\|v_1(t^*)-v_2(t^*) \|_2 \leq \|v_1(0)-v_2(0) \|_2e^{\alpha_1 t^*}=O\left((\varepsilon')^{1+\theta} \frac{\varepsilon_0}{\varepsilon'}\right)=O\left((\varepsilon')^{\theta} \right)
\end{align*}
and the error could still be arbitrarily small even after reaching level $\varepsilon_0.$
\end{heuristics}

Heuristic \ref{h:error_prop} works for the nonlinear dynamics too due to Lemma \ref{l:elsilon_0_error} showing
\begin{align*}
	\sup_{0 \leq t \leq t^*} \|u_1(t)-u_2(t) \|_2=O \left((\varepsilon')^{\theta} \right).	
\end{align*}
in accordance with \eqref{eq:epsilon_0_2}.

\begin{lemma}
\label{l:epsilon_0_level}
For  $0<\varepsilon' \leq \varepsilon_0$
$$u_1(t^*),u_2(t^*)=\varepsilon_0 \varphi+O\left(\varepsilon_0^{1+\theta} \right) \ \ \textit{in } L^2([0,1]), $$
hence, for small enough $\varepsilon_0>0$ 
$$\|u_1(t^*) \|_2, \|u_2(t^*) \|_2 \geq \frac{\varepsilon_0}{2}. $$
\end{lemma}

Note that Lemma \ref{l:epsilon_0_level} implies \eqref{eq:epsilon_0_2}.

\begin{proof}(Lemma \ref{l:epsilon_0_level})
	
	Let $u(t)$ be either $u_1(t)$ or $u_2(t).$ From \eqref{eq:linear_error2}
	\begin{align*}
		\left\|u(t^*)-v(t^*) \right\|_2 \leq \frac{\beta C}{\theta \alpha_1} \left(\| u(0)\|_2e^{ \alpha_1 t^*} \right)^{1+\theta}=O \left(\left(\varepsilon' \frac{\varepsilon_0}{\varepsilon'} \right)^{1+\theta} \right)=O\left(\varepsilon_0^{1+\theta}\right),
	\end{align*}
	making $u(t^*)=v(t^*)+O\left(\varepsilon_0^{1+\theta} \right).$
	
	\begin{align*}
		v(t^*)=&e^{\mathcal{A}t^*}u(0)=e^{\mathcal{A}t^*}\left(\varepsilon'\varphi_1+O\left((\varepsilon')^{1+\theta}\right) \right)\overset{e^{\mathcal{A}t}\varphi_1=e^{\alpha_1 t}\varphi_1}{=}\varepsilon_0\varphi_1+e^{\mathcal{A}t^*}O\left((\varepsilon')^{1+\theta}\right)) \\
		\overset{\|e^{\mathcal{A}t}\|_2=e^{\alpha_1 t}}{=}&\varepsilon_0\varphi_1+ O\left(e^{\alpha_1 t^*}(\varepsilon')^{1+\theta}\right)=\varepsilon_0\varphi_1+O\left(\varepsilon_0 (\varepsilon')^{\theta}\right)\\
		\overset{\varepsilon' \leq \varepsilon_0}{=}&\varepsilon_0\varphi_1+O\left(\varepsilon_0^{1+\theta}\right).
	\end{align*}
\end{proof}

Thus, \eqref{eq:epsilon_0_1} holds as well concluding the proof of Claim \ref{c:epsilon_0}.

\subsubsection{Getting $\varepsilon$-close to the endemic state}
\label{s:close_to_endemic}

We shift the time so that $\|u_1(0)-u_2(0) \|_2=O\left( (\varepsilon')^{\theta} \right)$ and $u_1(0),u_2(0)=\varepsilon_0 \varphi+O\left(\varepsilon_0^{1+\theta} \right).$

In this section we want to get the two solutions $u_1,u_2$ $\varepsilon$-close to the endemic state $\psi$. The deviations are denoted by $\tilde{u}_i(t):=u_i(t)-\psi$ ($i \in \{1,2\}$).

By lemma \ref{l:elsilon_0_error} the error until some time $t^{**}$ can be upper bounded as
\begin{align}
\label{eq:error_accumulation}
\sup_{0 \leq t \leq t^{**}}\|u_1(t)-u_2(t) \|_2 =O \left((\varepsilon')^{\theta}e^{\alpha_1 t^{**}} \right)
\end{align}  

which dependence on $t^{**}$. 
The question is, can we set a $t^{**}=t^{**}(\varepsilon)$ (but independent of $\delta,\varepsilon'$) such that $u_1,u_2$ are already close to equilibrium?

\begin{lemma}
\label{l:convergence_to_equilibrium}
If $\gamma>0$
\begin{align}
\label{eq:convergence_to_equilibrium}
\frac{\d}{\d t} \| \tilde{u}(t) \|_\pi^2 \leq -2\beta \int_{0}^{1} \pi(x) \left( \tilde{u}(t,x) \right)^2 \W u(t,x) \d x,
\end{align}
if $\gamma=0$
\begin{align}
	\label{eq:convergence_to_equilibrium'}
	\frac{\d}{\d t} \| \tilde{u}(t) \|_2^2 \leq -2\beta \int_{0}^{1}  \left( \tilde{u}(t,x) \right)^2 \W u(t,x) \d x.
\end{align}
\end{lemma}

\begin{corollary}
	Claim \ref{c:monotonicity}  is an easy consequence of Lemma \ref{l:convergence_to_equilibrium}. 
\end{corollary}

\begin{proof}(Lemma \ref{l:convergence_to_equilibrium})
	
	We start with the $\gamma=0$ case.
	\begin{align*}
	\partial_t \tilde{u}=&\partial_t (u-1)=\partial_t u=\beta(1-u)\W u=-\beta \left(\W u \right) \tilde{u} \\
	\frac{\d}{\d t} \|\tilde{u}(t) \|_2^2=& 2 \langle \partial_t \tilde{u}(t), \tilde{u}(t) \rangle=-2\beta \int_{0}^{1}  \left( \tilde{u}(t,x) \right)^2 \W u(t,x) \d x 
	\end{align*}
	
	Now assume $\gamma>0$.
	
	From \cite[Proposition 4.13.]{SISdyn} we know that the operator $\mathcal{B}=\beta(1-\psi)\W$ has spectral radius $\gamma$ with $\psi$ as the corresponding Perron-Frobenius eigenfunction. This means $\Lambda:=\mathcal{B}-\gamma \mathbb{I}$ has negative eigenvalues. Note that $\mathcal{B}$ is bounded in $L^2_{\pi}([0,1])$ since
	\begin{align*}
		\|\mathcal{B}f \|_{\pi}^2=&\int_{0}^{1} \pi(x) \beta^2 (1-\psi(x))^2 \left(\W f(x) \right)^2 \d x  \\
		\leq& \beta^2 \|\W f \|_2^2 \leq \beta^2 \|\W \|_2^2 \|f \|_2^2 \leq \beta^2 \|\W \|_2^2 \|f \|_\pi^2. \Rightarrow \|\mathcal{B} \|_\pi \leq \beta \|W \|_2.
	\end{align*}
	Furthermore, $\mathcal{B}$ is symmetric: in $L^2_{\pi}([0,1])$:
	\begin{align*}
	\langle \mathcal{B}f, g \rangle_{\pi}= \langle \beta \W f, g \rangle=	\langle  f, \beta \W g \rangle=\langle f, \mathcal{B}g \rangle_{\pi}.
	\end{align*}
	This makes $\Lambda$ selfadjoint in $L^2_\pi([0,1])$ with non-positive eigenvalues implying it is negative semi-definite.
	
	Next, we calculate the time derivative of $\tilde{u}.$
	\begin{align*}
		\partial_t \tilde{u}=&\partial_t (u-\psi)=\partial_t u=\beta(1-u)\W u-\gamma u\\
		=& \beta(1-\psi-\tilde{u})\W u-\gamma u=\Lambda u-\beta \left( \W u \right)\tilde{u}
	\end{align*}	
	Note that $\Lambda \psi=0$ implies $\Lambda u= \Lambda (\tilde{u}+\psi)=\Lambda \tilde{u}$, therefore, we end up with the expression
	\begin{align*}
		\partial_t \tilde{u}=\Lambda \tilde{u}-\beta \left( \W u \right) \tilde{u}.
	\end{align*}
	
	Finally, we give an upper bound on the derivative  of $\| \tilde{u}(t) \|_2^2.$
	\begin{align*}
		\frac{\d}{\d t} \| \tilde{u}(t) \|_\pi^2=& 2 \left \langle  \partial_t \tilde{u}(t),\tilde{u}(t) \right \rangle_{\pi}=2\left \langle \Lambda \tilde{u}(t),\tilde{u}(t) \right \rangle_\pi-2\beta\left \langle  (\W u(t)) \tilde{u}(t),\tilde{u}(t) \right \rangle_{\pi} \\
		\leq & -2\beta\left \langle  (\W u(t)) \tilde{u}(t),\tilde{u}(t) \right \rangle_{\pi}=-2 \beta \int_{0}^{1} \pi(x) \left( \tilde{u}(t,x) \right)^2 \W u(t,x) \d x
	\end{align*}
\end{proof}

\begin{remark}
\label{r:epsilon}
Assume $\gamma>0$. Since $u_1(0),u_2(0)$ have already reached the macroscopic level $\varepsilon_0$ we hope that some $\tilde{\varepsilon}_0>0$ exists such that $\W u(t,x) \geq \tilde{\varepsilon}_0,$ or in other words, every susceptible individual gets infected with a uniformly positive rate. Thus, by \eqref{eq:convergence_to_equilibrium} and Grönwall's lemma
\begin{align*}
\frac{\d}{\d t} \| \tilde{u}(t) \|_{\pi}^2 \leq& -2\beta \tilde{\varepsilon}_0 \int_{0}^{1} \pi(x) \left( \tilde{u}(t,x) \right)^2  \d x=-2\beta \tilde{\varepsilon}_0 \|\tilde{u}(t) \|_\pi^2 \\
\| \tilde{u}(t^{**}) \|_2 \leq & \| \tilde{u}(t^{**}) \|_\pi  \leq  e^{-\beta \tilde{\varepsilon}_0 t^{**}} \| \tilde{u}(0) \|_\pi \leq 2 \sqrt{\| \pi \|_1} e^{-\beta \tilde{\varepsilon}_0 t^{**}}=\frac{\varepsilon}{2}
\end{align*}
for
\begin{align}
\label{eq:t**}
t^{**}:=\frac{1}{\beta \tilde{\varepsilon}_0} \log \frac{\varepsilon}{4 \sqrt{\| \pi \|_1}}.
\end{align}

Note that, for this choice of $t^{**}$ the right hand side of \eqref{eq:error_accumulation} becomes dependent on $\varepsilon.$

Similarly, when $\gamma=0$ one has
\begin{align*}
	\|\tilde{u}(t^{**}) \|_2 \leq 2  e^{-\beta \tilde{\varepsilon}_0 t^{**}}=\frac{\varepsilon}{2} 
\end{align*}
for
\begin{align}
	\label{eq:t**'}
	t^{**}:=\frac{1}{\beta \tilde{\varepsilon}_0} \log \frac{\varepsilon}{4 }.
\end{align}
\end{remark}

The following lemma provides the existence of such $\tilde{\varepsilon}_0.$

\begin{lemma}
\label{l:epsilon_0_tilde}
Assume a), b) or c). Also, assume $u(0)=\varepsilon_0 \varphi_1+O \left(\varepsilon_0^2 \right)$ for some small enough $\varepsilon_0$. Then there exists an $\tilde{\varepsilon}_0>0$ such that for all $t \geq 0, \ x \in [0,1] $,
 $$\W u(t,x) \geq \tilde{\varepsilon}_0.$$
\end{lemma}

The proof of Lemma \ref{l:epsilon_0_tilde} will be handled in three different \emph{cases} according to whether we assumed a), b) or c) referred to as Case 1-3 respectively.

For Case 1, first we show that $u(t)$ can be bounded from below by a curve which is increasing and uniformly positive. The following lemma is a reformulation of \cite[Proposition 4.8.]{SISdyn}.

\begin{lemma}
	\label{l:monotone_lower_bound}
	Assume $\lambda_1 \beta>\gamma$ and $\varphi_1$ is bounded. Define $\varphi_\omega:=\omega \varphi_1$ and let $u_{\omega}$ be a solution to \eqref{eq:u} with initial condition $u_{\omega}(0)=\varphi_\omega.$
	
	Then there is a small $\omega_0>0$ such that for all $0<\omega \leq \omega_0$ $\varphi_\omega\in \Delta$ and $u_{\omega}(t,x)$ is monotone increasing in $t \geq 0.$
\end{lemma}

\begin{proof}(Lemma \ref{l:monotone_lower_bound})

	Since we are in the supercritical regime we can set a small $\epsilon>0$ such that $\mu:=\beta(1-\epsilon)\lambda_1-\gamma \geq 0.$ For such $\epsilon$ we can set a $\omega_0$ such that for all $0<\omega \leq \omega_0$ $0 \leq \varphi_{\omega}(x) \leq \epsilon$ making $\varphi_{\omega} \in \Delta.$
	
	Since $\varphi_{\omega}$ is an eigenvector of $\W$ with eigenvalue $\lambda_1$ we have
	\begin{align*}
		\beta(1-\epsilon)\W \varphi_{\omega}=\beta(1-\epsilon)\lambda_1 \varphi_{\omega}=(\gamma+\mu)\varphi_{\omega},
	\end{align*}
	implying
	\begin{align*}
		0 \leq \mu \varphi_{\omega}=\beta(1-\epsilon)\W \varphi_{\omega}-\gamma \varphi_{\omega} \leq  \beta(1-\varphi_{\omega})\W \varphi_{\omega}-\gamma \varphi_{\omega}.
	\end{align*}
	
	\cite[Proposition 2.12.]{SISdyn} implies $u_{\omega}(t)$ is increasing in $t \geq 0.$
\end{proof}

\begin{proof} (Case 1 of Lemma \ref{l:epsilon_0_tilde})
	
	For $x \in I_i$
	\begin{align*}
		u(0,x)= \frac{1}{|I_i|} \int_{I_i} u(0,y) \d y=\varepsilon_0 \phi_{i}^{(1)} +O \left( \varepsilon_0^2 \right) \geq \frac{\varepsilon_0}{2}\phi_{i}^{(1)}
	\end{align*}
	for small enough $\varepsilon_0$ uniformly in $i$. Therefore, $u(0,x) \geq \frac{\varepsilon_0}{2} \varphi_1(x).$
	
	In the discrete case,
	$$\|\varphi_1 \|_{\infty} = \max_{1 \leq i \leq n} \phi_{i}^{(1)}<\infty,$$
	making Lemma \ref{l:monotone_lower_bound} applicable. Set $\omega:=\min \{\frac{\varepsilon_0}{2},\omega_0\}.$ This makes $u(0) \geq u_{\omega}(0),$ hence $u(t) \geq u_{\omega}(t)$ for $t \geq 0$ as \eqref{eq:u} is cooperative (see \cite[Proposition 2.7.]{SISdyn}), and
	\begin{align*}
		\W u(t,x) \geq \W u_{\omega}(t,x) \geq \W u_{\omega}(0,x)=\W \varphi_{\omega}=\lambda_{1} \varphi_{\omega} \geq \lambda_1 \omega m=: \tilde{\varepsilon}_0>0.
	\end{align*}
	
\end{proof}

\begin{proof}(Case 2 of Lemma \ref{l:epsilon_0_tilde})
	
	Since $W$ is uniformly positive,
	\begin{align*}
		\W u(t,x)=&\int_{0}^{1} W(x,y)u(t,y) \d y \geq m_0 \int_{0}^{1}u(t,y) \d y \geq m_0 \int_{0}^{1}u^2(t,y) \d y  \\
		=& m_0 \|u(t) \|_2^2=m_0 \left( \|\varphi_1 \|_2 \|u(t) \|_2 \right)^2 \geq m_0 \langle \varphi_1, u(t) \rangle^2 =m_0 c_1^2(t),
	\end{align*}	
	thus, it is enough to have a lower bound for $c_1(t).$
	
	Since $u(0)=\varepsilon_0 \varphi_1+O\left(\epsilon_0^2 \right)$ one has $c_1(0) \geq \frac{\varepsilon_0}{2}$ for small enough $\varepsilon_0.$
	
	We will show that $c_1(t)$ can not go below a certain small positive value after reaching such level.
	
	Recall Lemma \ref{l:u_and_v}. From \eqref{eq:u_and_v}, taking the scalar product of both sides with $\varphi_1$ gives
	\begin{align*}
		c_1(t)=c_1(0)e^{\alpha_1 t}-\beta \int_{0}^{t}e^{\alpha_1(t-s)}\langle u(s) \W u(s), \varphi_1 \rangle \d s.
	\end{align*}
	
	The derivative is
	\begin{align}
		\label{eq:diff_c1}
		\begin{split}
			\frac{\d}{\d t}c_1(t)=&\alpha_1 c_1(0)e^{\alpha_1 t}-\beta \langle u(t) \W u(t), \varphi_1 \rangle-\alpha_1 \int_{0}^{t}e^{\alpha_1(t-s)}\langle u(s) \W u(s), \varphi_1 \rangle \d s \\
			=& \alpha_1 c_1(t)-\beta  \langle u(t) \W u(t), \varphi_1 \rangle.
		\end{split}	
	\end{align}
	
	Using \eqref{eq:Wu_upper_bound}
	\begin{align*}
		\langle u(t) \W u(t), \varphi_1 \rangle \leq  \frac{M}{m}c_1(t) \langle u(t) , \varphi_1 \rangle=\frac{M}{m}c_1^2(t),
	\end{align*}
	which gives
	\begin{align*}
		\frac{\d}{\d t} c_1(t) \geq c_1(t) \left( \alpha_1-\frac{\beta M}{m}c_1(t) \right).
	\end{align*}
	Assuming $\frac{\varepsilon_0}{2} \leq \frac{m \alpha_1}{\beta M} $ implies $c_1(t) \geq \frac{\varepsilon_0}{2}$ for all $t \geq 0$.		
\end{proof}

\begin{proof}(Case 3 of Lemma \ref{l:epsilon_0_tilde})
	
	Since $W$ is rank-$1$,
	\begin{align*}
		\W u(t,x)=\lambda_1 c_1(t) \varphi_1(x) \geq \lambda_1 m c_1(t),
	\end{align*}
	hence, it is enough to get a lower bound on $c_1(t)$.For this, we will use \eqref{eq:diff_c1} once again. As $Wu(t,x)=\lambda_1 c_1(t) \varphi_1(x)$ \eqref{eq:diff_c1} takes the form
	\begin{align*}
		\frac{\d}{\d t}c_1(t)=c_1(t) \left(\alpha_1-\beta \lambda_1 \left \langle u(t), \varphi_1^2 \right \rangle  \right).
	\end{align*} 
	
	Choose $K$ large enough so that $ \int_{\varphi_1(x) \geq K} \varphi_1^2(x ) \d x \leq \frac{\alpha_1}{2 \beta \lambda_1}.$
	\begin{align*}
		\beta \lambda_1 \left \langle u(t), \varphi_1^2 \right \rangle \leq \beta \lambda_1 \left \langle \1{\varphi_1 \geq K }, \varphi_1^2 \right \rangle +\beta \lambda_1 K \langle u(t), \varphi_1 \rangle  \leq \frac{\alpha_1}{2} +\beta \lambda_1 Kc_1(t),
	\end{align*}
	resulting in
	\begin{align*}
		\frac{\d}{\d t}c_1(t) \geq c_1(t) \left(\frac{\alpha_1}{2}-\beta \lambda_1 K c_1(t)  \right).
	\end{align*}
	
	The rest is similar to Case 2.
\end{proof}

Remark \ref{r:epsilon} with Lemma \ref{l:epsilon_0_tilde} shows \eqref{eq:epsilon_1} while \eqref{eq:error_accumulation} is the same as \eqref{eq:epsilon_2} concluding the proof of Claim \ref{c:epsilon_close}.

\subsection{Proofs regarding eternal solutions}
\label{s:proof_eternal}
\begin{proof}(Theorem \ref{t:existence})
	
	Set a small enough $\varepsilon_0>0$ whose value is chosen later. Also, define $\varepsilon_n:=\varepsilon_0 e^{-\alpha_1 n}.$ We will often use the identity
	$$e^{\alpha_1 n}= \frac{\varepsilon_0}{\varepsilon_n}. $$
	
	Let $(u_n)_{n=1}^{\infty}$ be a set of solutions with initial conditions
	\begin{align*}
		u_n(-n):=\min\{\varepsilon_n \varphi_1,1\}=\varepsilon_n \left(\varphi_1 -\left(\varphi_1-\frac{1}{\varepsilon_n}\right) \1{\varphi_1 \geq \frac{1}{\varepsilon_n}} \right)=:\varepsilon_n\left(\varphi_1-\eta_n \right).
	\end{align*}
	
	Clearly,
	\begin{align*}
		\| \eta_n \|_2^2=\int_{\varphi_1 \geq \frac{1}{\varepsilon_n}} \left(\varphi_1(x)-\frac{1}{\varepsilon_n}\right)^2 \d x 
	\end{align*}
	is monotone decreasing. To get a rate of convergence as $n \to \infty$, define the random variable $\xi$ on $[0,1]$ with density function $\varphi_1^2(x).$ Due to Markov's inequality,
	
	\begin{align*}
		\|\eta_{n} \|_2^2 \leq &\int_{\varphi_1 \geq \frac{1}{\varepsilon_n}} \varphi_1^2(x) \d x\\
		=& \pr \left(\varphi_1(\xi) \geq \frac{1}{\varepsilon_n}\right)=\pr \left(\varphi_1^{\rho}(\xi) \geq \varepsilon_n^{-\rho} \right) \leq \varepsilon_n^{\rho} \E \left( \varphi_1^{\rho}(\xi) \right)\\
		=& \varepsilon_n^{\rho} \int_{0}^{1}\varphi_1^{2+\rho}(x) \d x=O \left(\varepsilon_n^{\rho}\right) \\
		\|\eta_{n} \|_2=& O \left(\varepsilon_n^{\frac{\rho}{2}}\right).
	\end{align*}
	
	We will apply a Cauchy argument to show $u_n(t) \to u(t)$ for some $u$ which end up being a nontrivial eternal solution. Furthermore, for large enough $k_0 $ (and $n$) we have $u_n(-k_0) \approx \varepsilon_{k_0} \varphi_1.$
	
	Set $k_0$  and let $n,k$ be $n \geq k \geq k_0.$ 
	
	With a slight modification of \eqref{eq:linear_error2} in Lemma \ref{l:linear_error},

	\begin{align*}
		\sup_{0 \leq t \leq n-k_0}\|u_n(t-n)-v_n(t-n) \|_2 \leq& \frac{\beta C}{\theta\alpha_1}\|u_n(-n) \|_2^{1+\theta}e^{(1+\theta) \alpha_1(n-k_0)} \\
		\leq& \frac{\beta C}{\theta\alpha_1}\varepsilon_n^{1+\theta}e^{(1+\theta) \alpha_1(n-k_0)}= \frac{\beta C}{\theta \alpha_1}\varepsilon_{k_0}^{1+\theta}.	
	\end{align*}

	Define $c_l^{(n)}(t):=\langle \varphi_l, u_n(t) \rangle.$ To approximate $v_n(-k_0)$ first notice
	\begin{align*}
		c_l^{(n)}(-n)=\varepsilon_n\left(\delta_{1,l}- \langle \varphi_l,\eta_n \rangle \right),
	\end{align*}
	hence \eqref{eq:v_expanded} shows
	\begin{align*}
		v_n(-k_0)=&\sum_{l}c_l^{(n)}(-n)e^{\alpha_l(n-k_0)} \varphi_l=\varepsilon_ne^{\alpha_1(n-k_0)}\left(\varphi_1+O\left(\|\eta_n \|_2 \right) \right)\\
		=&\varepsilon_{k_0}\left(\varphi_1+O\left(\|\eta_{k_0} \| \right)\right) \ \textit{in} \ L^{2}([0,1]).
	\end{align*}
	 Combining the two bounds yields
	
	\begin{align*}
		\|u_n(-k_0)-\varepsilon_{k_0}\varphi_1 \|_2 \leq& \|u_n(-k_0)-v_n(-k_0) \|_2+\|v_n(-k_0)-\varepsilon_{k_0} \varphi_1 \|_2=\\
		&O \left(\varepsilon_{k_0}^{1+\theta}+\varepsilon_{k_0}\|\eta_{k_0} \|_2 \right) =O \left( \varepsilon_{k_0}^{1+\theta'} \right)
	\end{align*}
	where $\theta':=\min \{\theta, \frac{\rho}{2} \}.$

	The same argument can be used for index $k$, resulting in
	\begin{align}
		\label{eq:k_0}
		u_n(-k_0),u_{k}(-k_0)=\varepsilon_{k_0}\varphi_1+O\left( \varepsilon_{k_0}^{1+\theta'} \right) \ \textit{in} \ L^{2}([0,1])
	\end{align}
	which resembles the setup of Subsection \ref{s:epsilon_0}. Indeed, we can use Lemma \ref{l:elsilon_0_error} to show
	\begin{align}
		\label{eq:k_0_2}
		\begin{split}
		\sup_{0 \leq t \leq k_0}\|u_n(t-k_0)-u_k(t-k_0) \|_2  \leq& \|u_n(-k_0)-u_{k}(-k_0) \|_2e^{\alpha_1 k_0}=\\
		&O \left(\varepsilon_{k_0}^{1+\theta'} \frac{\varepsilon_0}{\varepsilon_{k_0}}\right)=O\left(\varepsilon_{k_0}^{\theta'} \right).
		\end{split}
	\end{align}
	Now, fix an arbitrarily large $\tau>0.$  Define the norm on $\mathcal{C}\left([-\tau,0], L^2([0,1]) \right)$ as
	\begin{align*}
		\|u_n-u_k \|_{[-\tau,0]}:=\sup_{-\tau \leq t \leq 0 } \|u_n(t)-u_k(0) \|_2.
	\end{align*}
	$\tau \leq  k_0$ for large enough $k_0$, hence \eqref{eq:k_0_2} yields
	\begin{align}
		\label{eq:Cauchy}
		\|u_n-u_k \|_{[-\tau,0]}=O\left(\varepsilon_{k_0}^{\theta'} \right) \to 0
	\end{align}
	as $k_0\to\infty$, making $(u_n)_{n=1}^{\infty}$ a Cauchy sequence for this norm, with limit $u$.
	
	Note that \eqref{eq:k_0} is also applicable for $k_0=0$ making $u_n(0)=\varepsilon_0 \varphi_1+O\left(\varepsilon_{0}^{1+\theta'} \right)$ in $L^2([0,1]).$ Choosing $\varepsilon_0$ small enough shows
	$$\frac{\varepsilon_0}{2} \|u_n(0) \|_2 \leq 2 \varepsilon_0, $$
	hence, $u$ can not be neither the disease-free nor the endemic state.
	
	$u(t)$ inherits being in $\Delta$ and satisfies \eqref{eq:u}, so it can be extended to $[-\tau, \infty[.$ To extend backwards in time, we consider $\tau'>\tau>0$ with corresponding limit $u'.$ Clearly $\left.u' \right|_{[-\tau,0]}=u$, otherwise $u_n(t) \to u(t) $ would not hold for $-\tau \leq t \leq 0.$ In conclusion, $u(t)$ can be extended to $\mathbb{R}$ as well making it a nontrivial eternal solution.
	
	When $W$ is discrete, $\varepsilon_0$ can be chosen small enough so that $u_n(-n)=\varepsilon_n \varphi_1 \in \Delta_I$ which property is also inherited by $u(t)$.
	
	Finally, we have to show \eqref{eq:small_eternal}. 
	
	Define
	$$\varepsilon_t:=\varepsilon_0e^{-\alpha_1 t}. $$
	
	Note that in the proof of \eqref{eq:k_0} we did not use the fact that $k_0$ is an integer -- $u_n$ is well defined for non-integer $n$ too -- hence, we also have
	\begin{align*}
		u_n(-t)=\varepsilon_t \varphi_1+O \left(\varepsilon_t^{1+\theta'} \right)  \ \textit {in } L^2([0,1]).
	\end{align*}

	This means
	\begin{align*}
		1 \geq \frac{\langle \varphi_1, u_n(-t) \rangle}{\|u_n(-t) \|_2}=\frac{\varepsilon_t+O \left(\varepsilon_t^{1+\theta'} \right)}{\varepsilon_t+O \left(\varepsilon_t^{1+\theta'} \right)} \geq 1-O \left(\varepsilon_{t}^{\theta'} \right)
	\end{align*}
	so after taking $n \to \infty$
	\begin{align*}
			1 \geq \frac{\langle \varphi_1, u(-t) \rangle}{\|u(-t) \|_2} \geq 1-O \left(\varepsilon_{t}^{\theta'} \right) \to 1.
	\end{align*}
\end{proof}

\begin{proof}(Lemma \ref{l:convergene_to_0})
	
	Let $u(t)$ be a nontrivial eternal solution and $\tilde{u}(t):=u(t)-\psi.$ Using Lemma \ref{l:convergence_to_equilibrium} for $-t \leq 0$,
	\begin{align*}
		4 \|\pi \|_1 \geq \| \tilde{u}(-t) \|_{\pi}^2 =&\|\tilde{u}(0) \|_{\pi}^2+\int_{-t}^{0} \left(- \frac{\d}{\d t} \|\tilde{u}(s) \|_{\pi}^2 \right) \d s \\
		\geq &\|\tilde{u}(0) \|_2^{\pi}+2 \beta \int_{-t}^{0}  \int_{0}^{1} \pi(x) \left(\tilde{u}(s,x)\right)^2 \W u(s,x) \d x \d s \\
		\geq & 2 \beta \int_{-t}^{0}  \int_{0}^{1}  \left(\tilde{u}(s,x)\right)^2 \W u(s,x) \d x \d s
	\end{align*}
	when $\gamma>0$ and
	\begin{align*}
		4 \geq & 2 \beta \int_{-t}^{0}  \int_{0}^{1}  \left(\tilde{u}(s,x)\right)^2 \W u(s,x) \d x \d s
	\end{align*}
	when $\gamma=0$.
	
	Since $\left(\tilde{u}(s,x)\right)^2 \W u(s,x) \geq 0$, the right hand side is monotone increasing in $t$ and bounded, hence
	\begin{align*}
		\int_{-\infty}^{0}  \int_{0}^{1} \left(\tilde{u}(s,x)\right)^2 \W u(s,x) \d x \d s<\infty,
	\end{align*}
	implying
	\begin{align*}
		\lim_{t \to -\infty} \int_{0}^{1} \left(\tilde{u}(t,x) \right)^2 \W u(t,x) \d x=0.
	\end{align*}
	
	Due to
	\begin{align*}
		\int_{0}^{1} \left(\tilde{u}(s,x)\right)^2 \W u(s,x) \d x \leq 4\int_{0}^{1}\int_{0}^{1}W(x,y) \d x \d y <\infty,
	\end{align*}
	we can take the limit inside:
	\begin{align*}
		\lim_{t \to -\infty}\left(\tilde{u}(t,x)\right)^2 \W u(t,x)=0.
	\end{align*}
	
	Define
	\begin{align*}
		u(-\infty,x):=\limsup_{t \to -\infty}u(t,x), \\
		A:=\{ x \in [0,1] \left| \W u(-\infty,x)=0 \right.\}.
	\end{align*}
	Clearly $u(-\infty,x)=\psi(x)$ for almost every $x \in A^c.$
	
	Note that, $A^{c}$ can not have measure $1$ as $u(0) \neq \psi$ and
	\begin{align*}
		\|\tilde{u}(-t) \|_\pi^2  \geq \|\tilde{u}(0) \|_\pi^2>0,
	\end{align*}
	when $gamma>0$ and
	\begin{align*}
		\|\tilde{u}(-t) \|_2^2  \geq \|\tilde{u}(0) \|_2^2>0,
	\end{align*}
	 when $\gamma=0$.
	
	Assume $A^c$ has measure strictly between $0$ and $1$. From the connectivity property \eqref{eq:W_connectivity} we have that
	$$D:=\int_{A} \int_{A^c}W(x,y) \d y \d x>0. $$
	
	When $W(x,y) \geq m_0>0$, we have
	\begin{align*}
		\psi(x)=\frac{\beta \W \psi(x)}{\gamma+\beta \W \psi(x)} \geq \frac{\beta m_0 \|\psi \|_1}{\gamma+\beta m_0 \|\psi\|_1 }>0.
	\end{align*}
	When $W$ is discrete,
	\begin{align*}
		\min_{x \in [0,1]}\psi(x)>0.
	\end{align*}
	based on \cite{NIMFA_bifurcation}, hence, in all cases we may assume $\psi$ is uniformly positive with lower bound $\psi_{\min}>0.$

	\begin{align*}
		0=&\int_{A} \W u(-\infty, x) \d x=\int_{A}\int_{0}^{1} W(x,y) u(-\infty,y) \d y \d x  \\
		\geq & \int_{A}\int_{A^c} W(x,y) u(-\infty,y) \d y \d x=\int_{A}\int_{A^c} W(x,y) \psi(y) \d y \d x \\
		\geq& \psi_{\min}  \int_{A}\int_{A^c} W(x,y)   \d y \d x = \psi_{\min} D>0,
	\end{align*}
	which leads to a contradiction.
	
	What remains is to show $\W u(-\infty)=0$ a.e. implies $u(-\infty)=0$ a.e. Define
	$$B_{\epsilon}:=\{ x \in [0,1] \left| u(-\infty,x) \geq \epsilon \right.\}. $$
	
	It is enough to show that $B_{\epsilon}$ has zero measure as $B_{\frac{1}{n}}$ are increasing sets. Note that $B_{\epsilon}$ can not have measure $1$ as that would lead to
	\begin{align*}
		0=\int_{0}^{1} \W u(-\infty,x) \d x \geq \epsilon \int_{0}^{1}\int_{0}^{1}W(x,y) \d y \d x>0.
	\end{align*}
	
	Assume $B_{\epsilon}$ has measure between $0$ and $1$. Then
	\begin{align*}
		0=&\int_{B_{\epsilon}^c} \W u(-\infty,x) \d x \geq \int_{B_{\epsilon}^c}\int_{B_{\epsilon}}W(x,y) u(-\infty,y) \d y \d x\\
		\geq & \epsilon\int_{B_{\epsilon}^c}\int_{B_{\epsilon}}W(x,y)  \d y \d x >0
	\end{align*}
	concluding the proof.
\end{proof}

\begin{proof}(Theorem \ref{t:uniqunes})
	
	Indirectly assume
	$$ \sup_{t \in \mathbb{R}} \inf_{\tau \in \mathbb{R}} \|u_1(t+\tau)-u_2(t) \|_2>0,$$
	that is, there exists a $t$ such that
	$$ \inf_{\tau \in \mathbb{R}} \|u_1(t+\tau)-u_2(t) \|_2>0. $$	
	
	With appropriate time translation we get 
	\begin{align}
		\label{eq:uniquenes_indirect}
		\inf_{\tau \in \mathbb{R}} \|u_1(\tau)-u_2(0) \|_2=:2 \varepsilon>0.
	\end{align}
	
	Set $\eta=\frac{\varepsilon}{2}$ and choose a corresponding $\delta>0$ according to Proposition \ref{t:main}.
	
	Due to Lemma \ref{l:convergene_to_0}, $\lim_{t \to -\infty}u_i(t)=0$ a.e. ($i=1,2$), hence, for $T \geq 0$ large enough,
	$$0<\|u_1(-T) \|_2, \|u_2(-T) \|_2 \leq \delta,$$
	meaning there must be times $t_1,t_2 \geq 0$ such that
	\begin{align*}
		\sup_{t \geq 0} \|u_1(t+t_1-T)-u_2(t+t_2-T) \|_2 \leq & \varepsilon, \\
		\sup_{0 \leq t \leq t_2} \|u_2(t-T) \|_2 \leq & \frac{\varepsilon}{2}.
	\end{align*}
	
	$T>t_2$ would give $\|u_1(t_1-t_2)-u_2(0) \|_2 \leq \varepsilon,$ violating \eqref{eq:uniquenes_indirect} with $\tau=t_1-t_2.$
	
	So $T \in [0,t_2]$, and then $\|u_2(0) \|_2=\|u_2(T-T) \| \leq \frac{\varepsilon}{2}. $
	
	Set $\tau$ to be a large negative number such that $\|u_1(\tau) \|_2 \leq \varepsilon.$ However, that implies 
	\begin{align*}
		2 \varepsilon \leq & \|u_1(\tau)-u_2(0) \|_2 \leq \|u_1(\tau) \|_2+\|u_2(0) \|_2 \leq \varepsilon+\|u_2(0) \|_2, \\
		\varepsilon \leq & \|u_2(0) \|_2,
	\end{align*}
	resulting in a contradiction.
\end{proof}

\section{Outlook}

In this paper we investigated the deterministic SIS process in general communities described by graphons, starting from small initial conditions. We have shown that after appropriate time translation, the solutions will be close to each other and identified their limit as the nontrivial eternal solution. This results in a huge reduction of complexity as one can neglect the exact initial configuration of infections as it can be well-approximated by the scaled version of the eigenvector centrality.

\rem{There are many questions left open for future work. Firstly, it seems reasonable to extend USIC to other compartmental models such as SIR or SEIR as the main ideas during the linearization phase is the same for these models. The main challenge is to show that convergence to the zero infection state is uniform in some sense. Although, the stationary solution is no longer unique, the differences are expected to be small if the curves stayed close together beforehand.

Secondly, lots of questions are remained open regarding the quantitative properties of the properties of eternal solution. Even innocent looking questions such is whether the eternal solution depends continuously on the graphon $W$ is yet to be answered.

Furthermore, it would be interesting to study USIC for sparse stochastic systems on an infinite graph, say the infinite $d$-regular tree and see what is their relation to the finite counterparts on random $d$-regular graphs. It is unclear at the moment whether it would be possible to define a non-trivial eternal contact process.

Lastly, it would be important to have a better understanding how small $\delta$ is required to be for different $W$, how practical USIC really is when it comes to real world applications.
}

\section*{Acknowledgment}
The author is thankful to Ill\'es Horv\'ath, P\'eter L. Simon and Istv\'an Z. Kiss for insightful discussions.

\label{s:outlook}

\bibliographystyle{abbrv}
\bibliography{USIC7bib}

\end{document}